\documentclass[12pt]{amsart}
\usepackage{amsmath}
\usepackage{amssymb}
\usepackage{hyperref}
\usepackage{amsfonts}
\usepackage{xypic}
\usepackage[left=3cm,right=3cm,top=2.5cm,bottom=2.5cm]{geometry}

\allowdisplaybreaks

\usepackage{physics}

\usepackage[T1]{fontenc}
\usepackage{amsmath}
\usepackage{amsfonts}
\usepackage{amssymb}
\usepackage{amsthm}
\usepackage{hyperref}
\usepackage{graphicx}
\usepackage{subcaption} 
\usepackage{euscript}
\usepackage{multicol}
\usepackage[dvipsnames]{xcolor}

\newcommand{\C}{\mathbb{C}}

\newcommand{\HH}{\mathbb{H}}
\newcommand{\N}{\mathbb N}
\newcommand{\R}{\mathbb R}
\newcommand{\Sph}{\mathbb S}

\newcommand{\sff}{\mathrm{I\!I}}
\newcommand{\Ric}{\text{Ric}}
\newcommand{\loc}{\text{loc}}
\newcommand{\dive}{\text{div}}
\newcommand{\p}{\partial}

\newtheorem{theorem}{Theorem}[section]
\newtheorem{lemma}[theorem]{Lemma}
\newtheorem{corollary}[theorem]{Corollary}
\newtheorem{definition}[theorem]{Definition}
\newtheorem{remark}[theorem]{Remark}

\numberwithin{equation}{section}

\begin{document}

 \title[Stable Minimal Hypersurfaces in 4-Manifolds]{Free Boundary Stable Minimal Hypersurfaces in Positively Curved 4-Manifolds}

\author{Yujie Wu}

\email{yujiewu@stanford.edu}

	\begin{abstract}
		We show that the combination of nonnegative 2-intermediate Ricci Curvature and strict positivity of scalar curvature forces rigidity of  two-sided free boundary stable minimal hypersurface in a 4-manifold with bounded geometry and weakly convex boundary.

		This extends the method of Chodosh-Li-Stryker to free boundary minimal hypersurfaces in ambient manifolds with boundary.
	\end{abstract}
	
	\maketitle

	\section{Introdution}
	Recall that a free boundary minimal hypersurface $(M,\p M)$ in a manifold $(X,\p X)$ with boundary is a critical point to the area functional among hypersurfaces whose boundary remains in the boundary of the ambient manifold. $M$ is called stable if its second variation is nonnegative among such hypersurfaces. Then given $\eta$ a choice of unit normal vector field along $M$, we have the following stability inequality for any compactly supported Lipschitz function $\phi$ over $M$,
	\begin{align*}
		\int_M |\nabla_M \phi|^2 \geq \int_M  (|\sff|^2+\Ric(\eta,\eta))\phi^2 +\int_{\partial M} A(\eta,\eta)\phi^2.
	\end{align*}

	In the case $\p M$ and $\p X$ are both empty, then nonnegative Ricci curvature of a closed ambient manifold forces rigidity results of its stable minimal hypersurfaces (Schoen-Yau \cite{SchoeYau1979-ExistenceIncompressibleMinimal},\cite{SchoeYau1979-StructureManifoldsPositive}); while if the ambient manifold is noncompact, to use the same method, we need to bound the volume growth of the minimal hypersurface.

	Chodosh-Li-Stryker \cite{ChodoLiStryk2022-CompleteStableMinimal} are able to use the method of $\mu$-bubble to give an almost-linear volume growth bound for a non-parabolic end of a minimal hypersurface in a noncompact 4-manifold with (suitable) positive curvature assumption.
	In this paper we study the analogous question for free boundary minimal hypersurfaces.

	We note that recently Catino-Mastrolia-Roncoroni \cite{CatinMastrRonco2023-TwoRigidityResults} has given ridigity results of complete stable minimal hypersurfaces in $\R^4$ or a positively curved Riemannian manifold $X^n$ when $n\leq 6$, where the authors look at a suitable positive curvature condition introduced in \cite{ShenYe1996-StableMinimalSurfaces}. 
	A review of progess in this direction can also be found in \cite{ChodoLiStryk2022-CompleteStableMinimal}.

	For an ambient 4-manifold $(X,\p X)$, we say $X$ has weakly convex boundary if the second fundamental form of the boundary is positive semi-definite. We use $\Ric, R$ to denote respectively Ricci and scalar curvature.  The so-called non-negative 2-intermediate Ricci curvature assmuption, denoted as $\Ric_2 \geq 0$, lies between non-negative sectional curvature and non-negative Ricci curvature, and will be explained in section \ref{setup}.

	\begin{theorem}
		Consider $(X^4, \partial X)$ a complete Riemannian manifold with weakly convex boundary, $R \geq 2$, $\Ric_2 \geq 0$, and weakly bounded geometry. Then any complete stable two-sided immersion of free boundary minimal hypersurface $(M,\partial M)\hookrightarrow (X,\partial X)$ is totally geodesic, $Ric(\eta,\eta)=0$ along $M$ and $A(\eta,\eta)=0$ along $\p M$, for $\eta$ a choice of unit normal over $M$.
	\end{theorem}
	In particular, any compact manifold $(X^4,\p X)$ with positive sectional curvature and weakly convex boundary will satisfy the assumption above. This gives the following nonexistence result:
	\begin{corollary}
		There is no complete two-sided stable free boundary minimal immersion in  a compact manifold $(X^4,\p X)$ with positive sectional curvature and weakly convex boundary.
	\end{corollary}  
	We will note two aspects that are mainly different from the case without boundary in \cite{ChodoLiStryk2022-CompleteStableMinimal} and require new ingredients. 

	The first is the notion of parabolicity and non-parabolicity for an end $E$ of manifolds with noncompact boundary, where we need to look at a (weakly) harmonic function $f$ with mixed (Dirichelt-Neumann) boundary conditions on two different parts of the boundary $\p E =\p_0 E \cap \p_1 E$. Standard ellipic regularity tells us that $f$ is smooth away from the points of intersection $\p _0 E \cap \p_1 E$.  By the work of Miranda \cite{Miran1955-SulProblemaMisto} we can see that $f$ is continuous (and bounded) around each point of intersection. Then the work of Azzam and Kreyszig \cite{AzzamKreys1982-SolutionsEllipticEquations} gives that if the interior angle of interesection $\theta$ is small, then $f$ is $C^{k,\alpha}$ for $k$ and $\alpha$ depending on $\theta$. This allows us to control the number of non-parabolic ends of $M$.
	\begin{theorem} 
		Let $(X^4,\partial X)$ be a complete manifold with $\Ric_2 \geq 0, A_2 \geq 0,$  and $(M,\partial M)$ a free boundary two-sided stable minimal immersion with infinite volume, then for any compact set $K \subset M$, there is at most 1 non-parabolic component in $M \setminus K$.
	\end{theorem}

	Here we write $A$ as the second fundamental form of $\p X$ in $X$, then $A_2 \geq 0$ is an intermediate assumption lying between convexity and mean convexity, which will be explained in Section \ref{setup}.

	The second ingredient is the bound of volume growth on a ball of fixed radius in $M$. In \cite{ChodoLiStryk2022-CompleteStableMinimal}, since $M$ has no boundary, with a uniform lower Ricci bound, we can obtain volume bound via Bishop-Gromov volume comparison theorem. 
	To apply the same for the free boundary case, one can exploit the assumption that $X$ has convex boundary. On the other hand, we can actually use the weakly bounded geometry assumption (that is already needed if one needs to apply blow-up argument to an arbitrary noncompact Riemann manifold).

	\begin{lemma}
        Let $(X^n,\partial X,g)$ be a complete Riemannian manifold with weakly bounded geometry at scale $Q$, and $(M^{n-1},\p M) \hookrightarrow (X,\p X)$ a complete immersed submanifold with uniformly bounded second fundamental form, then the following is true,
        \begin{itemize}
            \item there is $0<N<\infty$ such that for any $p\in M$, the maximum number of disjoint balls of radius $\delta$ centered around points in $B^M_{4\delta}(p)$ is bounded by N,
            \item for any $R>0$, there is a constant $C=C(R,Q)$ such that the volume of balls of radius $R$ around any point in $M$ is bounded by $C$.
        \end{itemize}
    \end{lemma}
	Proof of the lemma used an inductive covering argument in Bamler-Zhang \cite{BamleZhang2015-HeatKernelCurvature}. Preliminaries and outline of the paper is given in Section \ref{setup}.

	\textbf{Acknowledgements.} The author wants to thank Otis Chodosh for introducing this problem, for his continuous support and encouragement, and for many helpful discussions and comments on earlier drafts of this paper. The author wants to thank Richard Bamler for discussing the ideas of Lemma 2.1 in \cite{BamleZhang2015-HeatKernelCurvature}, Shuli Chen, Chao Li and Jared Marx-Kuo for interest in this work and comments on  the first draft. The author thanks Han Hong and Jia Li for pointing out typos in the first version of this paper in Theorem 6.4.

	\section{Set Up}\label{setup}

    We first set up some notations and definitions in this section.
    Recall that an immersed submanifold $M \hookrightarrow X$ is called minimal if its mean curvature vanishes everywhere. Throughout the paper we use the convention that $\sff_M(Y,Z)=-\langle \overline{\nabla}_Z Y, \nu_{M}\rangle$ given a choice of normal vector field $\nu_M$ for a hypersurface and $\overline{\nabla}$ the Levi-Civita connection on the ambient manifold $X$. We define mean curvature as $H_M=\text{tr}(\sff).$ In this convention, mean curvature of a sphere with outward unit normal in the Euclidean space is positive.

	In this paper, we also reserve the notation $\p M$ to denote the boundary of a continuous manifold instead of a subset.

    We first define the notion of free boundary minimal immersion. Consider an immersion of hypersurface $(M,\p M) \hookrightarrow (X,\p X)$, both manifolds with nonempty boundary, here we always require $\p M \subset \p X$ when writing $(M,\p M) \hookrightarrow (X,\p X)$. We write $\sff$ for the second fundamental form of $M \hookrightarrow X$  and $A$ for the second fundamental form of $\p X \hookrightarrow X$.  
    \begin{definition}
        We say that $M$ is a free boundary minimal immersed hypersurface if 
        \begin{itemize}
            \item the mean curvature $H=\trace(\sff)$ vanishes everywhere,
            \item $M$ meets $\p X$ orthogonally along $\p M$ (that is, the outward unit normal of $\p M$ agrees with the outward unit normal of $\p X$; so the second fundamental form of $\p M \hookrightarrow M$ is the same as restriction of the second fundamental form of $\p X \hookrightarrow X$ on $T\p M$).
        \end{itemize}
       
    \end{definition}

    The above definition has an equivalent characterization via variation of area.
 
    For any 1-parameter family of immersions $\varphi_t:(M,\p M) \hookrightarrow (X,\p X)$ with $t\in (-\epsilon,\epsilon)$, and $\varphi_0$ parametrizing $(M,\p M)\hookrightarrow (X,\p X)$, we write $V(x)=\frac{d}{dt}\vert_{t=0}(\varphi_t(M))$; we further require $V$ to be compactly supported along $M$, and $\p (\varphi_t(M))=\varphi_t(\p M)\subset \p X$, forcing $V$ to be parallel to the boundary of $X$ along $\p M$. Then the first  variation of area give,
    \begin{equation}
        \frac{d}{dt}\Big\vert_{t=0} \text{Area}(\varphi_t(M))=\int_M \dive_{M} V =-\int_M V\cdot H+\int_{\p M} V\cdot \nu_{\p M},
    \end{equation}
    with $\nu_{\p M}$ the outward unit normal of $\p M \hookrightarrow M$.

    Recall that an immersion $M\hookrightarrow X$ is called two-sided if there is a globally defined continuous unit normal vector field $\nu$.

    \begin{definition}
        A two-sided immersed free boundary minimal hypersurface $(M,\p M) \hookrightarrow (X,\p X)$ is stable if for any variation $\varphi_t$ (as defined above)  with vector field $V=f\nu$, the following stability inequality holds,
        \begin{align}
            0&\leq \frac{d^2}{dt^2}\Big\vert_{t=0} \text{Area}(\varphi_t(M))\\
            &=\int_M |\nabla_M f|^2-(|\sff|^2+\Ric_X(\nu,\nu))f^2-\int_{\p M}A(\nu,\nu)f^2.
        \end{align}
    \end{definition}

    We now introduce the curvature assumptions we made on the ambient manifolds. 
	The following curvature condition lies between nonnegative Ricci curvature and nonnegative sectional curvature (see also \cite{ChodoLiStryk2022-CompleteStableMinimal}). 
    \begin{definition}
        We say that $X$ has $\Ric_2 \geq 0$, i.e. nonnegative 2-intermediate Ricci curvature , if
	\begin{equation}
		R(v,u,u,v)+R(w,u,u,w)\geq 0,
	\end{equation} 
	for any $x\in X$ and any orthonormal vectors $u,v,w$ of $T_xM$, where $R(\cdot,\cdot,\cdot,\cdot)$ represents the Riemann curvature tensor of $X$. 
    \end{definition}
    \begin{remark}
        Note since $\Ric$ is symmetric, as long as the dimension of $X$ is  at least 3, $\Ric_2 \geq 0$ implies that $\Ric(u,u) \geq 0$ for any vector $u$ in the tangent plane of $X$ and so $\Ric \geq 0$ everywhere.
    \end{remark}

	Using $\Ric_2 \geq 0$ of the ambient manifold and Gauss Equation, we can control the Ricci curvature from below by the second fundamental form of a minimal immersion.
	\begin{lemma}[\cite{ChodoLiStryk2022-CompleteStableMinimal}, Lemma 2.2] \label{RicXtoM}
        Consider $(M^3,\p M) \hookrightarrow (X^4,\p X)$ immersed free boundary minimal hypersurface, if X has $\Ric_2 \geq 0$, then
		\begin{equation}
			\Ric_M \geq -|\sff|^2.
		\end{equation}
	\end{lemma}

	\begin{proof}
		Proof for interior points is the same as \cite{ChodoLiStryk2022-CompleteStableMinimal}; for boundary points we can extend by continuity.
	\end{proof}
	\begin{remark}
        The proof works in other dimensions too, the same conclusion holds for all $X^n$ with $n\geq 3$. When $n=3$, we would need $X$ to have positive sectional curvature. If $n\geq 4,$ we only need the following weaker assumption named $\Ric_{n-2}\geq 0$, meaning for any orthonormal vectors $e_1,...,e_{n-1}$ at a tangent plane of $X$, we have $$\sum_{k=2}^{n-1} R(e_k,e_1,e_1,e_k)\geq 0.$$
    \end{remark}

	We can in fact get a sharper bound with a constant depending on the dimension.

    \begin{lemma} [\cite{ChodoLiStryk2022-CompleteStableMinimal}, Lemma 4.2]   \label{RicXtoM2}
        Consider $(M^{n-1}, \p M) \hookrightarrow (X^n,\p X)$ immersed free boundary minimal  hypersurface, if $X$ has $\Ric_{n-2} \geq 0$, then
		\begin{equation}
			\Ric_M \geq -\frac{n-2}{n-1}|\sff|^2.
		\end{equation}
	\end{lemma}

	We also define an analogous ``2-convexity'' condition for $\partial X \hookrightarrow X$, lying between convexity and mean convexity.
	\begin{definition}
		For $(X,\partial X)$ a complete manifold with boundary, recall $A$ is the second fundamental form of $\partial X \hookrightarrow X$, we say that $A_2 \geq 0$ if for any orthonormal vectors $e_1, e_2$ on a tangent plane of $\p X$, we have $A(e_1,e_1)+A(e_2,e_2) \geq 0$.
	\end{definition}
    This condition will be useful combined with the stability inequality in section \ref{1NonparaEnd}.

	Also, to obtain blow up analysis needed for an arbitrary ambient Riemannian manifold, we require $(X,\partial X)$ to have weakly bounded geometry, defined as below.

	\begin{definition}\label{defWBG}
		We say a complete Riemannian manifold with boundary $(X,\partial X,g)$ has weakly bounded geometry (up to the boundary) at scale $Q$, if for this $Q>0,$ there is $\alpha \in (0,1)$ such that for any point $x \in X$, 
        \begin{itemize}
            \item there is a pointed $C^{2,\alpha}$ local diffeomorphism $\Phi: (B_{Q^{-1}}(a),a)\cap \HH_+ \rightarrow(U,x) \subset X,$ for some point $a \in \R^n$, here $\HH_+$ is the upper half space in $\R^n$;
            \item and if $\partial X \cap U \neq \emptyset$, then $\Phi^{-1}(\partial X \cap U) \subset \partial \HH_+$.
        \end{itemize}
		Furthermore, the map $\Phi$ has,
        \begin{itemize}
            \item $e^{-2Q} g_0\leq \Phi^{\star} g \leq e^{2Q} g_0 \text{ as two forms, with } g_0 \text{ the standard Euclidean metric;}$
            \item $\|\p_k \Phi^{\star} g_{ij}\|_{C^{\alpha}} \leq Q,$ where $i,j,k$ stands for indices in Euclidean space.
        \end{itemize}
	\end{definition}
    We will prove two consequences of this condition in the next section: one is the curvature estimates for stable free boundary minimal hypersurface following a result of Chodosh and Li \cite{ChodoLi2021-StableMinimalHypersurfaces},\cite{ChodoLi2022-StableAnisotropicMinimal}-- any two-sided complete stable minimal hypersurface in $\R^4$ is flat; the other is a volume control of balls of fixed radius by a constant depending on the coefficient $Q$ in the definition above.

	Until now we don't really need to restrict the ambient manifold to dimension 4. However, the dimension restriction is essential to the following theorem, where the $\mu-$bubble technique is needed to get a diameter bound using positive scalar curvature.

	\begin{theorem}
		Consider $(X^4,\p X)$ a complete manifold with scalar curvature $R\geq 2$, and $(M,\p M) \hookrightarrow (X,\p X)$ a two-sided stable immersed free boundary minimal hypersurface.
		Let $N$ be a component of $\overline{M \setminus K}$ for some compact set $K$, with $\p N =\p_0 N \cup \p_1 N, \p_0 N \subset \p M$ and $\p_1 N \subset K$. If there is $p\in N$ with $d_N(p,\p_1 N)>10\pi,$ then we can find a Caccioppoli set $\Omega \subset B_{10\pi}(\p_1N)$ whose reduced boundary has that: any component $\Sigma$ of $\overline{\p \Omega \setminus \p N}$ has diameter at most $2\pi$ and intersect with $\p_0 N$ orthogonally.
	\end{theorem}

	We will introduce the notion of Caccioppoli sets and the $\mu$-bubble technique in section \ref{mububble}.
    Now we can state our main theorem properly. 
	\begin{theorem}
		Consider $(X^4, \partial X)$ a complete Riemannian manifold with weakly convex boundary, $R \geq 2$, $\Ric_2 \geq 0$,  and weakly bounded geometry, then any complete stable two-sided immersion of free boundary minimal hypersurface $(M,\partial M)\hookrightarrow (X,\partial X)$ is totally geodesic, $Ric(\eta,\eta)=0$ along $M$ and $A(\eta,\eta)=0$ along $\p M$, for $\eta$ a choice of normal bundle over $M$.
	\end{theorem}

    \section{Weakly Bounded Geometry}
    We start with the first consequence, curvature estimates for free boundary stable minimal hypersurface in manifolds with weakly bounded geometry.
    \begin{lemma} \label{blowup}
		Let $(X^n,\partial X,g)$ be a complete Riemannian manifold with weakly bounded geometry, and $(M^{n-1},\p M) \hookrightarrow (X,\p X)$ a complete stable immersed free boundary minimal hypersurface, then
		\begin{equation*}
			\sup_{q\in M} |\sff(q)| \leq C <\infty,
		\end{equation*}
		for a constant $C=C(X,g)$ independent of $M$.
	\end{lemma}

    \begin{proof}
        We follow the proof as given in \cite{ChodoLiStryk2022-CompleteStableMinimal}.
		We prove that for any compact set $K \subset M$, we have the following curvature estimates:
		\begin{equation}
			\max_{q\in K} |\sff(q)|\min\{1, d_M(q,\partial_1 K)\} \leq C <\infty,
		\end{equation}
		with $\p M \cap K= \p_0 K$ and $\p K \setminus \p M= \p_1 K$,

        Towards a contradiction, assume there is a sequence of compact sets $K_i \subset M_i \hookrightarrow X$ the latter being a complete stable  immersed free boundary minimal hypersurface, and \begin{equation}
            \max_{q\in K_i} |\sff_i(q)|\min\{1,d_{M_i}(q,\p_1 K_i)\} \rightarrow \infty.
        \end{equation}
        Then by compactness of $K_i$ we can find $p_i \in K_i \setminus \p_1 K_i$ with
        \begin{equation}
            |\sff_i(p_i)|\min\{1,d_{M_i}(p_i,\p_1 K_i)\}=\max_{q\in K_i} |\sff_i(q)|\min\{1,d_{M_i}(q,\p_1 K_i)\} \rightarrow \infty.
        \end{equation}
        Define $r_i:=|\sff_i(p_i)|^{-1}\rightarrow 0$ and $x_i$ the image of $p_i$ in $X$. Using the weakly bounded geometry assumption and a pullback operation as in \cite{ChodoLiStryk2022-CompleteStableMinimal} Appendix B, we can find a sequence of pointed 3-manifolds $(S_i,s_i)$, local diffeomorphisms $\Psi_i:(S_i,s_i) \rightarrow (K_i,p_i)$ with the boundary components mapped correspondingly $\Psi_i (\p_l S_i)= \p_l K_i (l=0,1)$, and immersions $F_i:(S_i,s_i)\hookrightarrow (B(a_i,Q^{-1})\cap \HH_+,a_i)$ so that the following diagram commutes (writing $B_i:=B(a_i,Q^{-1})\cap \HH_+$), and that $F_i:S_i \rightarrow (B_i,\Phi^{\star}_i g)$ is a two-sided stable minimal immersion, in the free boundary sense along $\p_0 S_i$ but not $\p_1 S_i$, 
		\[
		\xymatrix{
		S_{i} \ar[r]^{F_{i}} \ar[d]_{\Psi_{i}}& B_i\ar[d]^{\Phi_{i}}\\
		M_i \ar[r] & X.
		}
		\]
		
		Note that in the weakly bounded geometry condition we may also require the Euclidean norm of $a_i$ is no more than $Q^{-1}$.

        We can now consider the blow-up sequence
        \begin{equation}
            \tilde{F_i}:(S_i,s_i) \rightarrow (\hat{B_i}, a_i), \,\, \hat{B_i}=B(a_i,r_i^{-1}Q^{-1})\cap \HH_+ \text{ with metric } 
			r_i^{-2}\Phi^{\star}_i g.
        \end{equation}
		By assumption of weakly bounded geometry, $(\hat{B_i}, a_i)$ converges to the Euclidean metric in $C^{1,\alpha}$ on any compact sets. We now consider $S_i$ with metric induced from $\tilde{F_i}$. By the point picking argument, for any point $q$ in a ball of fixed radius $R>0$ around $s_i$, we have a uniform bound on $|\tilde{\sff}_{S_i}(q)|\leq C(R)$. The weakly bounded geometry condition then gives $|\hat{\sff}_{S_i}(q)|\leq C'(R)$ for the immersion $\hat{F}:(S_i,s_i)\rightarrow (\hat{B_i},a_i)$, the latter with Euclidean metric $g_0$. This allows us to write a connected component of $B^{S_i}_{\mu}(q)$ as a graph of a function $f_i$ over a subset $B_r(0) \cap \HH_i$ of $T_q S_i$ for some $\mu,r >0$, here $B_r(0)$ is the Euclidean ball and $\HH_i$ is some halfspace in $\R^3$ that may not go through the origin. 
		
		Now following the same argument as in \cite{ChodoLiStryk2022-CompleteStableMinimal}, we know that the functions $f_i$ have uniformly bounded $C^{2,\alpha}$ norm. To continue the argument as in \cite{ChodoLiStryk2022-CompleteStableMinimal}, we can extend the graph $f_i$ from $B_r(0)\cap \HH_i$ to all of $B_r(0)$ and $f_i$ still has uniformly bounded $C^{2,\alpha}$ norm (but the extended part is not minimal as a hypersurface in $\hat{B_i}$). This gives us that on any bounded set, $(S_i,s_i)$ has injectivity radius bounded away from $0$ and bounded sectional curvature, with respect to the metric $(\tilde{F_i})^{\star} (r_i^{-2}\Phi_i^{\star}g)$.
		
		Then we can use the same argument in \cite{ChodoLiStryk2022-CompleteStableMinimal} and pass to the limit, to get  a subsequence converging to a complete minimal immersion $(S_{\infty},s_{\infty})$ in $\R^4$, or one that is minimal on $\HH_+$ and that intersect the $\p \HH_+$ orthogonally, furthermore $|\sff_{\infty}(s_{\infty})|=1$ (note that under this blow-up sequence, $\tilde{\sff_i}(s_i)=1$ by the choice of $r_i$ and $\tilde{d}_{S_i}(s_i,\p_1 S_i) \rightarrow \infty$). In the latter case we can use reflection principle (see for example Guang-Li-Zhou \cite{GuangLiZhou2020-CurvatureEstimatesStable}) and reduce to a complete minimal immersion in $\R^4$, which is a contradiction to the result of Chodosh and Li \cite{ChodoLi2021-StableMinimalHypersurfaces},\cite{ChodoLi2022-StableAnisotropicMinimal}-- any complete two-sided stable minimal hypersurface in $\R^4$ is flat.
    \end{proof}
	\begin{remark}
		The pullback operation in \cite{ChodoLiStryk2022-CompleteStableMinimal} applies to open manifolds without boundary(an interior ball of small radius in $K_i$ near $p_i$), in our case for the proof above, we need to extend over the free boundary part of this small ball, apply \cite{ChodoLiStryk2022-CompleteStableMinimal} to the extended open manifold and one can check that we still get a free boundary immersion near $\p_0 S_i$.
	\end{remark}

    Now we prove the following volume control theorem for a manifold with weakly bounded geometry. This argument follows as in Lemma 2.1 in Bamler-Zhang \cite{BamleZhang2015-HeatKernelCurvature}.
	In this paper given an immersion $M \hookrightarrow X$, we write the intrinsic distance function as $d_M(\cdot,\cdot)$ and extrinsic distance function as $d_X(\cdot,\cdot)$.
    \begin{lemma} \label{WBGvol}
        Let $(X^n,\partial X,g)$ be a complete Riemannian manifold with weakly bounded geometry at scale $Q$, and $(M^{n-1},\p M) \hookrightarrow (X,\p X)$ a complete immersed submanifold with bounded second fundamental form, then the following is true,
        \begin{itemize}
            \item there is $0<N<\infty$ such that for any $p\in M$, the maximum number of disjoint balls of radius $\delta$ centered around points in $B^M_{4\delta}(p)$ is bounded by N,
            \item for any $R>0$, there is a constant $C=C(R,Q)$ such that the volume of balls of radius $R$ around any point in $M$ is bounded by $C$.
        \end{itemize}
    \end{lemma}
    
    \begin{proof}
		To prove the first claim, we first prove that there is a fixed $0<r_0<Q^{-1}$ such that for any point  $p$ in $M$, we have for any $r<r_0$, $\Psi(B^S_r(s)) = B^M_r(p)$, here $\Psi$ comes from applying the pullback operation as in the previous lemma, i.e. we have the following commutative diagram, with local diffeomorphism $\Psi: (S,s) \rightarrow (M,p)$ and immersion $F:(S,s)\rightarrow (B,a)$, with $B=B_{Q^{-1}}(a) \cap \HH_+$,
		\[
		\xymatrix{
		(S,s) \ar[r]^{F} \ar[d]_{\Psi}& (B,a)\ar[d]^{\Phi}\\
		(M,p) \ar[r] & (X,x)
		}
		\]

		Note since image of any path in $B^S_r(s)$ is again a path in $B^M_r(p)$ and $\Psi$ is a local isometry, we have $\Psi(B^S_r(s)) \subset B^M_r(p)$. To prove the other direction, we look at a point $q$ connected to $p$ by a shortest path of unit speed $I(t):[0,l] \rightarrow M(l<r)$ , again since $\Psi$ is a local isometry we can find a path in $S$ with unit speed $J(t):[0,\epsilon] \rightarrow S, J(0)=s$, that is mapped isometrically to $I$ under $\Psi$. Writing Im$(I)$ for the image of $I(t)$ in M, we note that the preimage $\Psi^{-1}($Im$(I)$) is a union of paths in $S$ since $\Psi$ is a local isometry, one of the component must contain $J(t)$, which we denote as $J(t)$ from now on. The length of $J$ (denoted as $t_0$) is at least $l$, since if not, then as $t\rightarrow t_0$, $\Psi(J(t))$ converges to a point on the path $I(t)$, whose preimage in $J$ still lies in $B^S_r(s)$ and can be used to extend $J$ longer. Therefore $J$ must also reach a preimage of $q$ at length $l<r$.
		So we get $B^M_r(p)\subset\Psi(B^S_r(s))$.

		We now prove the first claim. Let $8\delta < r_0$, then we have that $\Psi(B^S_{4\delta}(s))=B^M_{4\delta}(p)$ by the above proof. For any disjoint balls $B^M_{\delta}(p_i)$ with $p_i \in B^M_{4\delta}(p)$, we must have $s_i \in B^S_{4\delta}(s)$, so that $\Psi(B^S_{\delta})(s_i)=B^M_{\delta}(p_i)$, therefore $B^S_{\delta}(s_i)$ are disjoint. 
	
		Note that $S \rightarrow B$ also has bounded second fundamental form, and the weakly bounded geometry assumption says the pullback metric via $\Phi$ is comparable to the Euclidean metric as two forms, which
		implies that the volume of $B^S_{5\delta}(s)$ is bounded above by $C \delta^{n-1}$, and the volume of $B^S_{\delta}(s_i)$ is bounded from below by $\C' \delta^{n-1}$ for some constant $C,C'$ depending on $Q$(here we may choose $r_0$ to be even smaller depending on the second fundamental form). Therefore, the number of such points $s_i$ is bounded by a fixed constant $N$, and so is the number of $p_i$.

		We now prove the second claim.
		We want to bound the volume of $B^M_R(p)$ for any given $R>0$ and any $p\in M$, and we may assume $R>r_0>8\delta$. Let $(B^M_{\delta}(p_i))_{i=1}^k$ be a choice of pairwise disjoint balls with centers in $B^M_{4\delta}(p)$ and with the maximum $k$ ($k \leq N$). By maximality, 
		\begin{equation}\label{bam1}
			B^M_{4\delta}(p)\subset \cup_{i=1}^k B^M_{2\delta}(p_i).
		\end{equation}
		We now argue that for all $r \geq 4\delta$, 
		\begin{equation}\label{bam2}
			B^M_{2\delta+r}(p)\subset \cup_{i=1}^k B^M_{r}(p_i).
		\end{equation}
		Consider a point $y \in B^M_{2\delta+r}(p)$, and a path $\gamma(t)$ (reparametrized by arc length) from $p$ to $y$ with length $l<r+2\delta$. Then by (\ref{bam1}) there is some point $p_i$ so that $\gamma(4\delta) \in \overline{B^M_{2\delta}(p_i)}$. We have,
		\begin{equation*}
			d(p_i,y) \leq l-4\delta+d(\gamma(4\delta),p_i) \leq l- 2\delta<r,
		\end{equation*}
		completing the proof of (\ref{bam2}).

		We now prove by induction that for any $k\geq 2$ and any $q \in M$, the volume of $B^M_{2\delta k}(q)$ is bounded by a constant $C^k$ with $C=C(Q,N,\delta)$. For $k=2$, this is already proved in the first claim. Now assuming the claim is true for some $k \geq 2$, then using (\ref{bam2}) for $r=2\delta k$ gives,
		\begin{equation*}
			|B^M_{2(k+1)\delta}(q)|\leq N C^k \leq C^{k+1}.
		\end{equation*} 
		Choosing $k$ large enough we can bound the volume of $B^M_R(q)$ for any given $R>0$. 
    \end{proof}

	\section{Parabolicity on Manifolds with Noncompact Boundary} \label{pnp}

	Given a manifold with boundary $(M^n,\p M)$, and any continuous submanifold $E^n$, recall we reserve the notation $\p E$ to denote the manifold boundary of $E$ (instead of as a subset in $M$). Therefore we can decompose $\p E =\p_1E \cup \p_0 E$ where $\p_0 E=\p E\cap \p M$ and $\p_1 E =\overline{\p E \setminus \p M}$. And we say that $\p_1 E \cap \p_0 E$ at an angle $\theta(x) \in (0,\pi)$, if for any $x \in \p_1 E \cap \p_0 E$ 
%	and any orthonormal basis of $T_x M$, 
	the hyperplane $T_x \p_1 E$ and $T_x \p_0 E$ intersect at angle $\theta(x)$ in the interior of $E$. 
%	Note the definition is independent of the choice of orthonormal basis. 
	In this paper we only consider domains $E$ in $M$ that are smooth except at the intersections $\p_1 E \cap \p_0 E$, we call these points corner points.

	\begin{definition}
		Consider $(M^n,\partial M)$ complete manifold with noncompact boundary. An end of $(M,\partial M)$ is a sequence of complete continuous n-dimensional submanifold $(E_k)_{k\geq 0}$ with boundary, where each $E_k$ is a noncompact connected component of $M\setminus C_k$ for compact continuous submanifold $C_k \subset C_{k+1}$, and $E_{k+1} \subset E_k$. 

		When $C_k=C_{k+1}=K$and $E_k=E_{k+1}=E$ for all $k \geq 0$, we will also call $E$ an end with respect to the compact set $K$.  
	\end{definition}

	\begin{definition}
		For any end $E$ of $M$, we  say that $\partial E $ intersect with the boundary $\partial M$ transversally (or at an angle $\theta(x) \in (0,\pi) $) if $\partial_0 E$ and $\partial_1 E$ intersect transversally (or at an angle $\theta(x) $) as submanifolds in $M$, that is, for any point $x \in \p_1 E \cap \p_0 E$, the tangent planes $T_x\p_0 E$ and $T_x\p_1 E$ are not equal (or at an angle $\theta(x)$).
	\end{definition}

	In the following theorem we show how we can purturb the angle of intersection of an end in an arbitrarily small neighborhood.

	\begin{theorem}\label{IntersectionAngle}
		Consider $(M^n,\partial M)$ complete orientable manifold with noncompact boundary and let $d_M(p,\cdot)$ be the continuous distance function from a fixed point $p\in M$ (we will mollify it to be smooth on a compact set in $M$ without changing the notation). Then for almost every $c>0$, the preimage $E_c=d^{-1}_M([c,\infty))$ is a submanifold with boundary and intersects with the boundary $\partial M$ transversally. Furthermore given any $\delta>0$ and constant $\theta \in (0,\frac{\pi}{2})$, we can find another continuous submanifold $E$ within the $\delta$-neighborhood of $E_c$ so that the angle between the tangent planes $T_x \p_1 E$ and $T_x \p_0 E$ is equal to $\theta.$ The submanifold $E$ is smooth except at the corners.
	\end{theorem}

	\begin{proof}
		We first consider the continuous distance function $h=d_M(p,\cdot)$, for any $N>0$ and any $\delta >0$, there is a mollification $\bar{h}$ such that $\bar{h}$ is smooth on $B^M_N(p)$ and $\|h-\bar{h}\|_{L^{\infty}(B^M_N(p))}<\delta/2$. Then it is a standard proof (see for example in \cite{GuillPolla1974-DifferentialTopology} section 2.1) that for almost every $0<c<N$, the map $d\bar{h}_x: T_xM \rightarrow \R$ and the map $d\bar{h}_x:T_x \p M \rightarrow \R$ are both nonzero, and the preimage $E_c=\bar{h}^{-1}([c,\infty)$ is a continuous submanifold intersecting $\p M$ transversally and is smooth except at the corners. We now show that we can purturb to arrange the angle of interesection to be any constant $\theta\in (0,\frac{\pi}{2})$ in a $\frac{\delta}{2}-$neighborhood of $E_c$.

		We denote the intersection $\partial M \cap \partial E_c=:I,$ note $I$ is orientable because it's the preimage of the regular value $s$ of the function $d(p,\cdot)$ restricted to the boundary by \cite{Lee2012-IntroductionSmoothManifoldsa} Proposition 15.23.
		Using a unit normal vector field $\mu$ of $I\subset \partial M$ that is outwarding pointing with respect to $E_c$, we find a local coordinates $(z,t)$ within the $\delta'-$neighborhood of $I \subset \partial M$($\delta'$ to be decided), here $(z,t)$ means $(z,0)\in I$ and $(z,t)$ stands for the point $\exp^{\p M}_{(z,0)}(t \mu)$(the exponential map on $\p M$). Now similarly using the outward pointing unit normal $\nu$ of $\partial M \subset M,$ we build a local coordinates denoted as $(z,t,r)=\exp^M_{(z,t)}(r\nu)$. Denote the projection map  onto the last coordinate $r$ as $P_r:E_c\rightarrow \R,$ then $0$ is a regular value of $P_r$ because if $dP_r(x): T_xE_s \rightarrow \R$ is zero for some point $x\in I=P_r^{-1}(0)$, then $T_x E_c \subset T_x \partial M$, contradiction to the transversal intersection of them we just proved. Further note $dP_r$ is zero restricted to $T_x \partial M,$ espeically in the directions on $T_x I.$ Now fix a point $(z_0,0,0) \in I,$ and consider the slice $S_{z_0}=\{(z,t,r)\in E_c,z=z_0\}$ in the $rt-$plane, then $P_r$ restricted to $S_{z_0}$ has that $dP_r$ is nonzero around a neighborhood of origin, so the tangent line along $S_{z_0}$ is never parallel to the $t$-axis in this neighborhood, meaning we can write $S_{z_0}$ as a graph $(z_0,t(r),r)$ in this neighborhood (the function $t(r)=t_{z_0}(r)$ also depends on $z_0$ but we omit the notation). 

		Now we can concatenate the graph $t(r)$ with the linear map $\bar{t}(r)=\tan(\theta) r$, at $r=\delta''$ for some $\delta''<\delta'$, to get a new function $\hat{t}(r)$ with jump singularity at $r=\delta''$, and using a bump function $\phi(r)$ supported near the singularity, we have the function $\hat{t}(r)(1-\phi(r))$  gives the graph bounding our desired E together with $E_c$. Given a fixed $\theta \in (0,\frac{\pi}{2})$, we can choose $\delta',\delta''$ small enough so that the modification happens within the $\frac{\delta}{2}-$neighborhood of $E_c$.
	\end{proof}

	From now on, in this section we will mostly follow the discussion in \cite{ChodoLiStryk2022-CompleteStableMinimal} where the case is for manifolds without boundary.

	\begin{definition}[Parabolic Component] 
		Let $(M^n,\p M)$ be a complete Riemannian manifold with noncompact boundary, and $E$ an end with respect to some compact $K$.
		We say that $E$ is parabolic if there is no  positive harmonic function $f \in C^{2,\alpha}(E)$, for some $\alpha>0$, so that, 
		\begin{equation*}
			f\big|_{\partial_1 E}=1,\quad \partial_{\nu}f\big|_{\partial_0 E}=0, \quad f\big|_{E^{\circ}}<1,
		\end{equation*}
		with $\nu$ the outward pointing unit normal of $\p M$.

		Otherwise we say that $E$ is nonparabolic.
	\end{definition}

	We note that if $E$ is nonparabolic, then there is a harmonic function $f$ on $E$ that is $C^{2,\alpha}$ across the corners, in the sense that it can be extended to an open neighborhood of $E$ in $M$. 

	We first deal with the regularity issue arising in the above definition. That is, when $\p_1 E \cap \p_0 E \neq \emptyset$, a weakly harmonic function over $E$ may not lie in the class $C^2(E)$ or even $C^1(E)$. The following theorem says that if we purturb the angle of intersection of $\p_1 E \cap \p_0 E$ to be small, we will have enough regularity.

	\begin{theorem} \label{regularity}
		Consider a connected compact Riemannian manifold with boundary $(K,\partial K =\partial_1 K \cup\partial_0 K)$, and $\partial_1 K$ intersect with $\partial_0 K$ transversally as smooth codimension 1 submanifolds, with constant angle $\theta \in (0,\pi/4)$ contained in $K$. We write $\nu$ as the outward pointing unit normal at each boundary ($\nu$ exists almost everywhere, i.e. except at the corner points). Then a weakly harmonic function $f \in W^{1,2}(K)$ with prescribed boundary condition: $f\vert_{\partial_1 K}=g\vert_{\partial_1 K}$, and $\nabla_{\nu}f\vert_{\partial_0 K}=\nabla_{\nu}g\vert_{\partial_0 K}$ with $ g\in C^{2,\alpha(\theta)}(K)$, is also  $C^{2,\alpha(\theta)}$ for some fixed $\alpha(\theta)>0.$
	\end{theorem}
	\begin{proof}
		The function $u=g-f$ satisfies $\Delta u= \Delta g=:h$ and has Dirichlet boundary condition over $\partial_1 K$ and Neumann boundary condition over $\partial_0 K$. Then $u$ is the unique solution to the following problem, in a  complete subspace of $W^{1,2}(K)$, namely
		$$\int_K \nabla u \cdot \nabla \phi= - \int_K h\phi, \quad \forall \phi \in C^{\infty}_c(K\setminus \p_1 K),$$
		\text{over the set }$\mathcal{S}:=\left\{ u \in W^{1,2}(K), u\vert_{\p_1 K}=0\right\}.$

		We note that a unique solution exists by Lax-Milgram, and we have that the $W^{1,2}$ norm of the solution $u$ is finite since,
		\begin{equation*}
			\int_K \nabla u \cdot \nabla u =-\int_K hu \leq \|h\|_{L^2} \|u\|_{L^2} \leq C \|h\|_{L^2} \|\nabla u\|_{L^2},
		\end{equation*}
		where in the last step we used Poincar\'e inequality since $u \vert_{\partial_1 K}=0$ ($\partial_1 K \neq \emptyset$).
		So away from the corners we can continue with standard iteration scheme (see for example \cite{Evans2010-PartialDifferentialEquations}, \cite{AmbroCarloMassa2018-LecturesEllipticPartial} and \cite{GilbaTrudi2001-EllipticPartialDifferential}) to get for any $k\in \N$, $\|u\|_{H^k} \leq C' \|u\|_{H^1} \leq C(h,K)$, where $\|u\|_{H^k}:=\|\nabla^k u\|_{L^2(K)}$. We briefly write the process using partition of unity here.

		Given any interior ball $B_r\subset B_R \subset K^{\circ}$ consider a bump function supported on $B_R$  and $\phi=1$ on $B_r$. Then $\Delta (\phi u)=(\Delta \phi) u + 2\nabla u \cdot \nabla \phi + h\phi \in L^2$, so we have that $\|\phi u\|_{H^2} \leq C' (\|\Delta (\phi u)\|_{L^2}+\|u\|_{H^1})\leq C(R,r,h) \|u\|_{H^1}$. Differentiating the equation again and iterate the process, we get the claimed bounds on $H^k$ norm of $u$ on $B_r$. So we can get $C^{\infty}_{loc}$ bounds on any compact set in the interior.

		A similar process holds if $B_r\subset B_R$ are balls centered around a boundary point $B_R \cap \partial_0 K = \emptyset$. Consider $\phi f$ with $\phi$ compactly supported in $B_R$ but is equal to 1 on $B_r$ (including points on the boundary), look at $\phi u$ on $B_R$ (and flatten the intersetion of $B_R$ and $\partial_1 K$, this is not an issue since we only want to bound $u$ in $B_r$). Then the same process as above applies using boundary estimates. 
		
		For purely Neumann condition a similar treatment holds. We need to choose bump functions $\phi$ supported in boundary coordinates charts, so that on the boundary of $B_R$, $\phi=1, \partial_{\nu}\phi=0$, to make sure $\partial_{\nu} (\phi f)=0$ on the boundary of $B_R$ (again we flatten the intersection of $B_R$  and $\partial_0 K$). Then using boundary estimates for Neumann conditions, we again have the above property.

		If $B_R$ is a ball centered around a point on the corners: $\partial_1 
		K \cap \partial_0 K$, we have $\Delta  u=h$ on $B_R$, using normal coordinates for small $r$, the function $u$ solves a uniformly elliptic nonhomogeneous equation, both in the weak sense and classically everywhere except at the corners. We choose a smooth bump function $\phi$ like in the Neumann case, i.e. $\phi=1$ and $\p_{\nu} \phi=0$ on the boundary of $B_R$. Then by the work of Miranda \cite{Miran1955-SulProblemaMisto},  $u\phi$ is (H\"older) continuous (and bounded)  on $B_R$, and under this assumption, using the method of barrier functions, Azzam  \cite{azzam1981smoothness} gives that $u \in C^{2,\alpha(\theta)}(B_r)$ for $\theta \in (0,\pi/4)$. 
%		Writing $\delta(x)=d(x,\partial_1 M \cap \partial_0 M)$, 
		The following bounds holds on $B_r$ for $r <\frac{R}{2}$ (\cite{azzam1981smoothness}):
		\begin{equation} \label{regbound}
			\sup_{x\in B_r} |u(x)|+\sup_{x,y\in B_r}\frac{|D^2u(x)-D^2 u(y)|}{d_K(x,y)^{\alpha}}\leq C
%			\delta(x)|Du(x)| + \delta^2(x)|D^2 u(x)| \leq C(\delta(x))^{2+\alpha(\theta)},
		\end{equation}
		where the constant only depend on the manifold $K$, the function $g$ and the constant $\alpha$.
		In particular, on any compact set in $K$, $u$ has bounded $C^{2,\alpha}$ norm and so does $f$, i.e. $\|f\|_{C^{2,\alpha}}(K)\leq C(g,K)$. We will make use of the bound soon.
	\end{proof}

	\begin{remark}
		The book of Miranda \cite{Miran1970-PartialDifferentialEquations}, the paper of Liebermann \cite{Liebe1986-MixedBoundaryValue} and of Azzam and Kreyszig \cite{AzzamKreys1982-SolutionsEllipticEquations} give a nice review  over regularity of solutions of mixed boundary value problem. 
	\end{remark}

	In this paper, when we say that an end is parabolic or non-parabolic, we always mean that $\p_1 E\cap \p_0 E$ with a constant angle in $(0,\pi/4)$. Applying Hopf Lemma (see \cite{GilbaTrudi2001-EllipticPartialDifferential} Lemma 3.4) we have the following maximum principle.

	\begin{theorem}\label{MaxiP}
		If $K$ is compact in $M$, and $f$ is harmonic on $K$ with $\p_{\nu} f \vert_{\p_0 K}=0$, then
		\begin{equation*}
			\max_{\p_0 K} f \leq \max_{\p_1 K}f, \quad \min_{\p_0 K} f \geq \min_{\p_1 K} f.
		\end{equation*}
		In particular, $\max_{K} f=\max_{\p_1 K} f$ and $\min_{K} f=\min_{\p_1 K} f$.
	\end{theorem}

	\begin{lemma}\label{parabolic0energy}
		Let $(M,\partial M)$ be a complete Riemannian manifold. Let $K\subset M$ be a compact subsest of $M.$ Let $E\subset M$ be an unbounded component of $M\setminus K$, fix $p\in E$ and consider $B_{R_i}(p)$. Assume $E$ is parabolic, then there are positive harmonic functions $f_i$ on $E\cap B_{R_i}$ with 
		$$f_i\rvert_{\partial_1 E}=1, \nabla_{\nu}f_i\rvert_{\partial_0 E}=0,  f_i \rvert_{\partial_1 B_{R_i}}=0,$$
		with $R_i \rightarrow \infty$. Then $f_i \rightarrow 1$ in $C^{2,\alpha}_{\rm{loc}}(E)$ and $\lim_i \int_E |\nabla f_i|^2=0.$
	\end{lemma}

	\begin{remark}
		Again we may choose $R_i$ and mollify the boundary $\partial B_{R_i} \cap \partial M$ without relabeling so that the angle of intersection is $\theta \in (0,\pi/4)$. We will omit this step later when mollification is needed.
	\end{remark}

	\begin{proof}
		Let $f_i$ be the minimizer of Dirichlet energy over $B_{R_i}$ given the above boundary conditions.
		We first claim that $f_i$ has finite and decreasing Dirichlet energy. Since given a Lipschitz domain in $\R^n$,  a function is in $W_0^{1,2}$ (zero trace) if and only if it can be approximated by a sequence of compactly supported smooth functions, and $E$ has Lipschitz boundary, using a partition of unity, the same  holds for on $E$. So if we extend $f_1$ by zero on $B_{R_i}\setminus B_{R_1}$ we get another candidate and that we may assume $\int_E |\nabla f_{i+1}|^2 \leq \int_E |\nabla f_i|^2\leq  \int_E |\nabla f_1|^2=C_1.$

		Using Lemma \ref{regularity} and maximum principle, we know that $\|f_i\|_{C^0}\leq 1$, for all $i\geq 0$. Now using equation (\ref{regbound}), we know that $ \sup_i  \|f_i\|_{C^{2,\alpha}(K')}$ is finite for any compact subset $K' \subset E$. 
		
		We also have that $f_i$ subsequentially converge in $C^{2,\alpha}_{loc}$ (for some $\alpha >0$) to a harmonic function $1\leq f \leq 0$ on $E$, and by parabolicity and maximum principle, $f=1$ everywhere on $E$, and we have:
		\begin{equation*}
			\int_E |\nabla f_i|^2 = \int_{\partial_1 E} f_i \nabla_{\nu} f_i \rightarrow 0,
		\end{equation*}
		using the uniform convergence to $f=1$ in $C^1_{\loc}-$norm near $\p_1 E$.
	\end{proof}

	We note that nonparabolicity is inherited by subsets. The proof of the lemma below is analogous to Proposition 3.5 in \cite{ChodoLiStryk2022-CompleteStableMinimal}  if we use Lemma \ref{regularity} to deal with regularity of mixed boundary value problem.
	\begin{lemma}
		Consider $K\subset \hat{K}$ compact subset in $(M,\p M)$, with each component of $M \setminus \hat{K}$ and $M \setminus K $ is smooth except at the corners. If $E$ is a nonparabolic component of $M \setminus K$, then there is a nonparabolic component of $M \setminus \hat{K}$.
	\end{lemma}

	The above lemma, together with {Theorem \ref{IntersectionAngle}}  says that, starting with any nonparabolic end $E_1:=E \subset M\setminus K$, we can build a sequence of nonparabolic sets $E_k$ with $\partial_1 E_k\cap \partial M$ contained correspondingly in any small neighborhood of $\partial B_{R_k}(p),$  intersecting with $\partial M$ at angle $\theta$ for any $\theta \in (0,\pi/4),$ for any $R_k$ in a open dense set of $(0,\infty).$ Hence we have the following definition.
	
	\begin{definition}[Nonparabolic Ends]
		Let $(E_k)$ be an end with each $\partial E_k$ intersecting with $\partial M$ at angle $\theta \in (0,\pi/4)$ and smooth except at the corners, we say that $(E_k)$ is a nonparabolic end if $k\geq 0$, the component $E_k$ is nonparabolic.
	\end{definition}

	We also note that the unique minimal barrier function on a nonparabolic end has finite Dirichelt energy, a fact we will use in Section \ref{1NonparaEnd}.

	\begin{theorem}\label{finiteDE}
		If $E$ is a nonparabolic end of $M$, then there is a harmonic function $f$ over $E$ with $f \vert_{\partial_1 E}=1$ and $\nabla_{\nu} f\vert_{\partial_0 E}=0$, that is minimal among all such harmonic functions and has finite Dirichlet energy.
	\end{theorem}

	\begin{proof}
		By definition of nonparabolicity, there is a positive harmonic function $g$ with $g\vert_{\partial_1 E}=1, \partial_{\nu} g\vert_{\partial_0 E} =0$. We solve over an exhaustion $\cup_{i\in \N}\Omega_i = E$, the following mixed boundary value problem (each $\Omega_i$ contains $\partial_1 E$),
		\begin{align*}
			\Delta f_i=0, \quad f_i\vert_{\partial_1 E}=1, \quad \partial_{\nu} f_i\vert_{\partial_0 E}=0, \quad f_i\vert_{\partial \Omega_i \setminus (\partial_1E \cup \partial_0 E)}=0.
		\end{align*}
		We may assume all the corners of $\Omega_i$ has interior angle in $(0,\pi/4)$. Maximum principle then gives that $f_i \leq g$ over $\Omega_i$.
		Using the same argument as in Lemma \ref{parabolic0energy}, we have that $f_i$ converge in $C^{2,\alpha}_{\loc}$ to a positive barrier function over $E$, that is bounded by $g$. Since this argument applies for arbitrary $g$, we have that $f$ is the unique minimal barrier function. Now we show $f$ has finite Dirichlet energy.
		\begin{align*}
			\int_{\Omega_i} |\nabla f_i|^2 = \int_{\partial_1 E} f_i\nabla_{\nu} f_i \leq C_0,
		\end{align*}
		where the last inequality is bounded by a constant we again used equation (\ref{regbound}) near a compact set containing $\p_1 E$. Now we can let $i\rightarrow \infty$ in the equation below to get that $f$ has finite Dirichlet energy.
		\begin{align*}
			\int_{\Omega_i} |\nabla f|^2 =\lim_{l>i} \int_{\Omega_i} |\nabla f_l|^2 \leq C_0.
		\end{align*}

	\end{proof}

	\section{At Most One Nonparabolic End}\label{1NonparaEnd}

	We follow the same method in \cite{ChodoLiStryk2022-CompleteStableMinimal} to show that under a suitable condition ($A_2 \geq 0$) for the boundary $\p X$ of an ambient manifold $X$ with $\Ric_2 \geq 0$, any  free boundary stable minimal hypersurface with infinite volume can only have at most 1 nonparabolic end. We begin with the following theorem.

	\begin{theorem}\label{HarFinEne}
		Consider $(M,\partial M)$ a complete manifold, $K \subset M$ compact and $E_1, E_2$ are two nonparabolic components of $M\setminus K$. Then there is a  nonconstant bounded harmonic function with finite Dirichlet energy on $M.$
	\end{theorem}

	\begin{proof}
		By definition of parabolicity, on each end $E_s(s=1,2)$ we can find a harmonic function $1\geq h_s(x)>0$ with $h_s|_{\partial_1 E} =1, \partial_{\nu} h_s \vert_{\partial_0 E}=0$. Using Lemma \ref{finiteDE}, we may assume that each $h_s$ has finite Dirichlet energy.
		
		We solve for harmonic functions $f_i$ on $B_{R_i}$ (again mollifying the boundary to get small intersection angle with $\partial M$) such that $f_{\partial_1 B_{R_i} \cap E_1}=h_1$, $f_{\partial_1 B_{R_i} \cap E_2}=1-h_2$, $f_i=0$ on other components of $\partial_1 B_{R_i}$, and $\partial_{\nu}f_i \vert_{\partial_0 M} =0$. Using a similar argument to that in section 4, we have that $$\sup_i\|\nabla f_i\|^2_{L^2(B_{R_i})}\leq C(\|\nabla f_1\|^2_{L^2(B_{R_1})}+\|\nabla h_1\|^2_{L^2}+\|\nabla h_2\|^2_{L^2})<\infty,$$ and that $f_i$ converges in $C^{2,\alpha}_{\loc}$ to a harmonic function on $M$ with finite Dirichlet energy. The function takes value in $[0,1]$ by maximum principle, and is nonconstant by arrangement at the two ends $E_1, E_2$. 
	\end{proof}

	\begin{theorem} \label{stableineq1}
		Let $(X^4,\partial X)$ be a complete manifold with $\Ric_2 \geq 0,$ and $(M^3,\partial M)$ a free boundary orientable stable minimal immersion, given a smooth harmonic function $u$ on M with Neumann boundary condition, we have the following estimates:
		\begin{align*}
			&\frac{1}{3}\int_M \phi^2 |\sff|^2 |\nabla u|^2 + \frac{1}{2}\int_M \phi^2 |\nabla |\nabla u||^2 \\
			\leq & \int_M |\nabla \phi|^2 |\nabla u|^2
			+ \int_{\partial M} |\nabla u| \nabla_{\nu} |\nabla u| \phi^2 - A(\eta,\eta)|\nabla u|^2\phi^2.
		\end{align*}
		Here $\sff$ is the second fundamental form of $M \rightarrow X$ and $A$ is the second fundamental form of $\partial X \rightarrow X$, $\nu \perp T\partial M$ in $TM$ and $\eta \perp M$ in $X$.

		If we have $A_2 \geq 0$, then:
		\begin{equation} \label{stabineq2}
			\frac{1}{3}\int_M \phi^2 |\sff|^2 |\nabla u|^2 + \frac{1}{2}\int_M \phi^2 |\nabla |\nabla u||^2 \leq \int_M |\nabla \phi|^2 |\nabla u|^2
		\end{equation}
	\end{theorem}

	\begin{proof}

		Using the second variation for orientable hypersurfaces we have for any family of immersion with speed $\frac{d}{dt}\big\vert_{t=0} \varphi_t(M)=\phi \eta$:
		\begin{align*}
			0&\leq \frac{d^2}{dt^2}\Big\vert_{t=0} \text{Area}(\varphi_t(M))\\
			&=\int_M |\nabla_M \phi|^2 - (|\sff|^2+\Ric(\eta,\eta))\phi^2 -\int_{\partial M} A(\eta,\eta)\phi^2
		\end{align*}

		Fixing any compact supported smooth function $\phi,$ we plug in $\sqrt{|\nabla u|^2 +\epsilon}\phi$ to the second variation formula then let $\epsilon \rightarrow 0$ to get the following,
		\begin{align*}
			0 &\leq \int_M |\nabla \phi|^2|\nabla u|^2 + \phi^2 |\nabla |\nabla u||^2+ \langle \nabla \phi^2, \nabla |\nabla u| \rangle |\nabla u| - |\sff|^2|\nabla u|^2 \phi^2 -\int_{\partial M} |\nabla u|^2 A(\eta,\eta)\phi^2\\
			&=\int_M |\nabla \phi|^2|\nabla u|^2 -|\nabla u|\Delta |\nabla u|\phi^2- |\sff|^2|\nabla u|^2 \phi^2 +\int_{\partial M} \phi^2(-|\nabla u|^2 A(\eta,\eta)+|\nabla u|\nabla_{\nu}|\nabla u|),
		\end{align*}
		here we have used that $\Ric_2 \geq 0$ implies $\Ric_X \geq 0$.
		Note over $M^{\circ}$ we have the following (see also \cite{ChodoLiStryk2022-CompleteStableMinimal}):
		\begin{align*}
			\Delta |\nabla u|^2 = 2 \Ric(\nabla u,\nabla u) + 2 |\nabla^2 u|^2&, \quad \text{Bochner's Formula}\\
			|\nabla^2 u|^2\geq \frac{3}{8} |\nabla u|^{-2} |\nabla |\nabla u|^2|^2&, \quad \text{Improved Kato's Inequality}\\
			\Ric(\nabla u,\nabla u)  \geq \frac{-2}{3} |\sff|^2|\nabla u|^2&, \quad \text{Lemma \ref{RicXtoM2}}
		\end{align*}
		These together imply $|\nabla u|\Delta |\nabla u| \geq \frac{-2}{3} |\sff|^2 |\nabla u|^2 +\frac{1}{2}|\nabla |\nabla u||^2,$ which we can plug into the last inequality, to get:
		\begin{equation*}
			\int_M \frac{1}{3} |\sff|^2 |\nabla u|^2 \phi^2 +\frac{1}{2}|\nabla |\nabla u||^2 \phi^2 \leq \int_M |\nabla \phi|^2 |\nabla u|^2 +\int_{\partial M} \phi^2(|\nabla u| \nabla_{\nu} |\nabla u|- |\nabla u|^2 A(\eta,\eta) ).
		\end{equation*}

		Note using Neumann condition we get $0=\nabla_{\nabla u}\langle \nabla u,\nu \rangle=\langle \nabla_{\nabla u} \nabla u,\nu \rangle+\langle \nabla u, \nabla_{\nabla u} \nu \rangle.$
		So we can compute the boundary terms:
		\begin{align*}
			\int_{\partial M} & |\nabla u| \nabla_{\nu} |\nabla u| \phi^2 - A(\eta,\eta)|\nabla u|^2\phi^2\\
			= \int_{\partial M} & - |\nabla u|^2(\langle \frac{\nabla u}{|\nabla u|},\nabla_{\frac{\nabla u}{|\nabla u|}} \nu\rangle + A(\eta,\eta))\phi^2)\\
			= \int_{\partial M} & - |\nabla u|^2(\langle \frac{\nabla u}{|\nabla u|},\nabla_{\frac{\nabla u}{|\nabla u|}} \nu\rangle+ \langle \eta, \nabla_\eta \nu \rangle)\phi^2
		\end{align*}
		Using that $A(e_1,e_1)+A(e_2,e_2) \geq 0$ if $e_1 \perp e_2$ (note $\eta \perp M$ while $\nabla u$ is along $M$), the above integrand over the boundary is now nonnegative, and  we have the inequality:
		\begin{equation*}
			\frac{1}{3}\int_M \phi^2 |\sff|^2 |\nabla u|^2 + \frac{1}{2}\int_M \phi^2 |\nabla |\nabla u||^2 \leq \int_M |\nabla \phi|^2 |\nabla u|^2
		\end{equation*}
	\end{proof}

	\begin{theorem} \label{1npend}
		Let $(X^4,\partial X)$ be a complete manifold with $\Ric_2 \geq 0,$ and the boundary of $X$ has second fundamental form satisfying $A_2 \geq 0,$  and $(M,\partial M)$ a free boundary orientable stable minimal immersion with infinite volume, then for any compact set $K \subset M$, there is at most 1 nonparabolic component in $M \setminus K$. In particular, M has at most one non-parabolic end.
	\end{theorem}

	\begin{proof}
		Since we can apply inequality (\ref{stabineq2}) of Theorem (\ref{stableineq1}), we have, for any compactly supported smooth function $\phi,$
		\begin{equation*}
			\frac{1}{3}\int_M \phi^2 |\sff|^2 |\nabla u|^2 + \frac{1}{2}\int_M \phi^2 |\nabla |\nabla u||^2 \leq \int_M |\nabla \phi|^2 |\nabla u|^2.
		\end{equation*}
		We can proceed as in \cite{ChodoLiStryk2022-CompleteStableMinimal}. Suppose there are two nonparabolic components $E_1, E_2$, we can find a nonconstant harmonic function with finite Dirichelt energy and Neumann boundary condition on $M$ by Theoerem \ref{HarFinEne} . We build the cut-off function based on $\rho(x):$ fix $z\in M$, $\rho$ is a mollification of $d_M(\cdot, z)$ such that $\rho \vert_{\partial B_{R_i}(z)} =R_i$ and $|\nabla \rho| \leq 2$. The cut-off $\phi_i(x)$ is equal to $1$ in $B_{R_1}(z)$, it's equal to $0$ outside $B_{R_i}(z)$ and equal to $\frac{R_i -\rho(x)}{R_i-R_1}$ otherwise (we may assume  $B_{R_1}(z) \subset K$).

		Then using $\phi_i$  we have as $R_i \rightarrow \infty$:
	\begin{align*}
		\frac{1}{3}\int_M \phi_i^2 |\sff|^2 |\nabla u|^2 + \frac{1}{2}\int_M \phi_i^2 |\nabla |\nabla u||^2 \leq \int_M |\nabla \phi_i|^2 |\nabla u|^2 \leq \frac{4\int_M |\nabla u|^2}{(R_i-R_1)^2} \rightarrow 0.
	\end{align*}

	So we get that  $|\sff| |\nabla u|=0=|\nabla |\nabla u||$ over $B_{R_1}(z)$, and letting $R_1 \rightarrow \infty$ gives us the two terms vanish on $M$. So $|\nabla u|$ is constant, and using $u$ has finite Dirichlet energy on $M$ which has infinite volume, we must have $\nabla u=0$, a contradiction since $u$ is nonconstant. 
	\end{proof}

	\section{$\mu$-bubble and Almost Linear Volume Growth} \label{mububble}

	We begin with some background on Caccioppoli sets used in our setting for free boundary $\mu$-bubbles. One can find preliminaries of Caccioppoli sets or $\mu$-bubble in \cite{Maggi2012-SetsFinitePerimeter}, \cite{ChodoLiStryk2022-CompleteStableMinimal}.

	\begin{definition}
		A measurable set $\Omega$ in a compact Riemannian manifold $N^l$ is called a Caccioppoli set (or a set of finite perimeter) if its characteristic function $\chi_{\Omega}$ is a function of bounded variation, i.e. the following is finite:
		\begin{align*}
			P(\Omega):=\sup\Big\{\int_\Omega \dive(\phi), \phi \in C^1_0(N^{\circ},\R^l), \|\phi\|_{C^0} \leq 1\Big\},
		\end{align*}
		We call $P(\Omega)$ the perimeter of $\Omega$ inside $N$(it's also equal to the $BV-$norm of $\chi_{\Omega}$ inside $N$).
	\end{definition}

	Using Riesz Representation theorem, the distributional derivative $\nabla (\chi_{\Omega})$ is a Radon measure and we can find a Borel set (up to change of zero measure) whose topological boundary is equal the support of this measure (see \cite{Maggi2012-SetsFinitePerimeter}). We will always assume $\Omega$ is such a set and use $\p \Omega$ to denote its reduced boundary. We note in \cite{Maggi2012-SetsFinitePerimeter} the reduced boundary is denoted as $\p^{\star} \Omega$ and is contained in the topological boundary, by De Giorgi's structure theorem the $l-1$ dimensional Hausdorff measure of $\p^{\star} \Omega$ is equal to $P(\Omega)$. The next lemma establishes regularity of $\p \Omega$ for minimizers of an appropriate functional.

	Consider a compact Riemannian manifold $N^3$ with boundary $\partial N =\partial_0 N \cup \partial_- N \cup \partial_+ N$ ($\partial_i N$ is nonempty for $i\in\{0,-,+\}$), where $\partial_- N$ and $\partial_+ N$ are disjoint and each of them intersect with $\partial_0 N$ at angles no more than $\pi/8$ inside $N$ . We fix a smooth function $u>0$ on $N$ and a smooth function $h$ on $N\setminus (\partial_- N \cup \partial_+ N)$, with $h\rightarrow \pm\infty$ on $\partial_{\pm} N$. We pick a regular value $c_0$ of $h$ on $N\setminus (\partial_- N \cup \partial_+ N)$ and pick $\Omega_0=h^{-1}((c_0,\infty))$. We want to find a minimizer among Caccioppoli sets for the following functional:
	\begin{equation}
		\mathcal{A}(\Omega):=\int_{\partial \Omega} u-\int_N (\chi_{\Omega}-\chi_{\Omega_0})hu.
	\end{equation}

	\begin{lemma}[Existence of Minimizers]
		There is a minimizer $\Omega$ for the above functional $\mathcal{A}$. The minimizer has smooth boundary which intersects with $\partial_0 N$ orthogonally. Also $\Omega \triangle\Omega_0$ is a compact subset in $N^{\circ} \cup \partial_0 N$.
	\end{lemma}

	\begin{proof}
		We can take $\Omega_0$ as a candidate so the infimum value of $\mathcal{A}$ is finite. Now we take a minimizing sequence $\Omega_k$. 
		Using approximate identity $\varphi_{k_j}$, we have $\chi_{\hat{\Omega}}-\chi_{\Omega_0}:=\chi_{\Omega} \star \varphi_{k_j}-\chi_{\Omega_0}$ converges in $L^p(p\geq 1)$ to $\chi_{\Omega}-\chi_{\Omega_0},$ together with that the $BV-$norm is lower semicontinuous with respect to $L^1-$norm, we can apply mollification to assume each $\chi_{\Omega_k}$ has smooth boundary. 

		Now note that since $h\rightarrow \pm\infty$ on $\partial_{\pm} N$, we may assume each $\Omega_k$ contains some fixed small neighborhood $\Omega_{\tau,+}$ of $\partial_{+} N$ and must not contain some fixed small neighborhood $\Omega_{\tau,-}$ of $\partial_{-} N$ for a $\tau>0$ (this is proved in details in \cite{ChodoLiStryk2022-CompleteStableMinimal}) Proposition 12. 
		
		So the function $(\chi_{\Omega_k}-\chi_{\Omega_0})hu$ is supported on the compact set $N \setminus \Omega_{\tau,\pm}$ and uniformly bounded in $k$, 
		since there is some $\delta>0$ so that $u>\delta>0$ on $N \setminus \Omega_{\tau,\pm}$ and $\Omega_k$ is a minimizing sequence, we get that the BV-norm of $\Omega_k$ is uniformly bounded in k, and so a subsequence converge in the following sense: $\nabla(\chi_{\Omega_k})$ in the weak$^{\star}$ sense as Radon Measures, $\chi_{\Omega_k}$  in the $L^1$ sense, and the limit $\chi_{\Omega_{\infty}}$ is also a BV function. Therefore $\mathcal{A}(\Omega_{\infty})=\lim_k \mathcal{A}(\Omega_{k}),$ and we found a minimizer.

		We note that regularity of free boundary minimal hypersurfaces has been established by Jost-G\"ruter \cite{GruteJost1986-AllardTypeRegularity}, corresponding to the case $u=1$ and $h=0$ for the functional $\mathcal{A}$. For general ellipic integrand and almost minimizers, by De Philippis and Maggi \cite{DePhMaggi2014-RegularityFreeBoundaries} Theorem 1.10 or \cite{DePhMaggi2014-DimensionalEstimatesSingular} Theorem 1.5, we have that $\p \Omega$ is a $C^{1,\frac{1}{2}}$ hypersurface in N interesecting with $\p_0 N$ orthogonally by the first variation formula given below, which also gives us the mean curvature (exists weakly \textit{a priori}) is a smooth function, this gives smoothness of $\p \Omega$.
	\end{proof}

	We now compute the first and second variation for such minimizers.

	\begin{theorem} \label{VarForm}
		Assume $\Omega$ is a minimizer of $\mathcal{A}$ in the settings above, we have the following first  variation formula, writing $\Sigma =\p \Omega$,
		\begin{equation*}
			\nabla_{\nu_{\Sigma}} u-hu+uH_{\Sigma} =0 \text{ on } \Sigma, \quad \nu_{\partial \Sigma} (x) \perp T_x\partial N \quad \text{for }x \in \partial \Sigma \subset \partial_0 N ,
		\end{equation*}
		and the second variation formula,
		\begin{align*}
			&\frac{d^2}{dt^2}\Big\vert_{t=0}(\mathcal{A}(\varphi_t(\Omega)))\\
			=&\int_{\Sigma} |\nabla_{\Sigma} \phi|^2u  -\frac{u\phi^2}{2}(R_{N}-R_{\Sigma}+|\sff|^2+H^2_{\Sigma})+ \phi^2 (\Delta_N u-\Delta_{\Sigma}u-\nabla_{\nu_{\Sigma}} (hu))\\
			& -\int_{\partial \Sigma} u  \phi^2 A(\nu_{\Sigma},\nu_{\Sigma})\\
			\leq & \int_{\Sigma} |\nabla \phi|^2u  -\frac{u\phi^2}{2}(R_{N}-R_{\Sigma})+\phi^2(\Delta_N u-\Delta_{\Sigma}u-u\nabla_{\nu_{\Sigma}} h-\frac{h^2u}{2}-\frac{u^{-1}}{2}(\nabla_{\nu_{\Sigma}}u)^2)\\
			& -\int_{\partial \Sigma} u  \phi^2 A(\nu_{\Sigma},\nu_{\Sigma})
		\end{align*}
	\end{theorem}

	\begin{proof}
		The computation follows similarly from first and second variation formula of free boundary minimal hypersurfaces.
		We consider a family of diffeomorphism $\varphi_t$ of $N$ with vector field $X_t$, notice that if $x \in \partial N$ then $X_t \in T_x \partial N$. Let $\partial \Omega_t =:\Sigma_t$, the first variation is given as:
		\begin{align*}
			\frac{d}{dt}(\mathcal{A}(\varphi_t(\Omega)))&=\frac{d}{dt} \int_{\partial \Omega_t} u - \frac{d}{dt}\int_{\Omega_t} hu\\
			&=\int (\frac{d}{dt}u)dvol_{\partial \Omega_t}+ \int u \frac{d}{dt} dvol_{\partial_{\Omega_t}} - \int hu \frac{d}{dt} dvol_{\Omega_t}\\
			&= \int (\nabla_{X_t} u) dvol_{\partial \Omega_t} -\int (hu) \langle X_t, \nu_{\partial \Omega_t}\rangle dvol_{\partial \Omega_t} + \int (u \dive_{\partial \Omega_t} X_t) dvol_{\partial \Omega_t} \\
			&= \int (\nabla_{X_t} u) dvol_{\partial \Omega_t}  -\int (hu)\langle X_t, \nu_{\partial \Omega_t}\rangle  dvol_{\partial \Omega_t}+\int (u \dive_{\partial \Omega_t} X_t) dvol_{\partial \Omega_t}\\
			&=\int (\nabla_{X_t} u) dvol_{\Sigma_t}  -\int (hu) \langle X_t, \nu_{\Sigma_t}\rangle dvol_{\Sigma_t}+ \int (u \dive_{\Sigma_t} (X^{\perp}_t+X^{\top}_t)) dvol_{\Sigma_t}\\
			&=\int (\langle \nabla u, (X_t-X^{\top}_t)\rangle - hu\langle X_t, \nu_{\Sigma_t}\rangle-u\vec{H_{\Sigma_t}}\cdot X^{\perp}_t ) dvol_{\Sigma_t}+\int_{\partial \Sigma_t} u \langle X_t, \nu_{\partial \Sigma_t}\rangle \\
			&= \int (\nabla u \cdot X^{\perp}_t -hu\langle X_t, \nu_{\Sigma_t}\rangle-u\vec{H_{\Sigma_t}}\cdot X^{\perp}_t ) dvol_{\Sigma_t} + \int_{\partial \Sigma_t} u \langle X_t, \nu_{\partial \Sigma_t}\rangle
		\end{align*}

		Note we used the convention that mean curvature $H$ is defined as the trace of the second fundamental form and hence $H_{\Sigma}=-\langle\nabla_{e_i} e_i,\nu_{\Sigma}\rangle=-\vec{H_{\Sigma}}\cdot \nu_{\Sigma}$.
		So at $t=0$ we have $\nabla_{\nu_{\Sigma}} u-hu+uH_{\Sigma} =0$ on $\Sigma,$ and that $\nu_{\partial \Sigma} (x) \perp T_x\partial N$ for $x \in \partial \Sigma \subset \partial N.$

		Now we continue with the second variation :
		\begin{align*}
			&\frac{d}{dt}\Big\vert_{t=0}(\mathcal{A}'(\varphi_t(\Omega))-\int_{\partial \Sigma_t} u \langle X_t, \nu_{\partial \Sigma_t}\rangle)\\
			=& \int \frac{d}{dt}\Big\vert_{t=0}(\nabla u \cdot X^{\perp}_t -hu\langle X_t, \nu_{\Sigma_t}\rangle-u\vec{H_{\Sigma_t}}\cdot X^{\perp}_t ) dvol_{\Sigma}\\
			=&\int \phi_t \frac{d}{dt} \Big\vert_{t=0}(\nabla u \cdot \nu_{\Sigma_t} -hu+uH_{\Sigma_t})dvol_{\Sigma}\\
			=& \int \phi_t (\partial_t \langle\nabla u,  \nu_{\Sigma_t}\rangle-\nabla_{X_t}(hu)+ (\nabla_{X_t} u) H_{\Sigma_t}+ u \partial_t H_{\Sigma_t}) dvol_{\Sigma},  \quad \text{at $t=0$}.\\
		\end{align*}
			Since at $t=0, \partial \Sigma \cap \partial N$ orthogonally, using the exponential map near $\partial \Sigma$, for any smooth function $\phi_t$, the diffeomorphism near $\Sigma$ given by $\Sigma \times (-\epsilon,\epsilon) \ni (x,t) \rightarrow \exp_x(t\nu_x)$ is admissible and produce a normal variation near $\Sigma$.
			We will also use $\Delta_N u-\Delta_{\Sigma}u=\nabla^2 u(\nu_{\Sigma},\nu_{\Sigma})+H_{\Sigma}\nabla_{\nu_{\Sigma}}u$:
		\begin{align*}
			&\frac{d}{dt}\Big\vert_{t=0}(\mathcal{A}'(\varphi_t(\Omega))-\int_{\partial \Sigma_t} u \langle X_t, \nu_{\partial \Sigma_t}\rangle)\\
			=&\frac{d}{dt}\Big\vert_{t=0}(\mathcal{A}'(\varphi_t(\Omega)))\\
			=& \int \phi_t^2(\nabla^2u(\nu_{\Sigma_t},\nu_{\Sigma_t})-\nabla_{ \nu_{\Sigma_t}}(hu)+H_{\Sigma_t}\nabla_{ \nu_{\Sigma_t}}u) +\phi_t \langle\nabla u,\p_t \nu_{\Sigma_t}\rangle dvol_{\Sigma}\\
			&+ \int u \phi_t(-\Delta_{\Sigma_t} \phi_t-\phi_t(|\sff_{\Sigma_t}|^2+\Ric_N(\nu_{\Sigma_t},\nu_{\Sigma_t})))dvol_{\Sigma}  ,  \quad \text{at $t=0$} \\
			=& \int_{\Sigma} |\nabla_{\Sigma} \phi|^2u  -u\phi^2(|\sff_{\Sigma}|^2+\Ric_N(\nu_{\Sigma},\nu_{\Sigma}))+ \phi^2 \nabla^2 u(\nu_{\Sigma},\nu_{\Sigma})+\phi_t \langle\nabla u,\p_t \nu_{\Sigma_t}\rangle\\
			&-\int_{\Sigma} \phi^2(\nabla_{ \nu_{\Sigma}}(hu)-H_{\Sigma}\nabla_{ \nu_{\Sigma}}u)+\int_{\Sigma} \langle \nabla_{\Sigma} u, \nabla_{\Sigma} \phi \rangle \phi   -\int_{\partial \Sigma} u\phi \nabla_{\nu_{\partial \Sigma}} \phi\\
			=&\int_{\Sigma} |\nabla_{\Sigma} \phi|^2u  -u\phi^2(|\sff_{\Sigma}|^2+\Ric_N(\nu_{\Sigma},\nu_{\Sigma}))+ \phi^2 (\Delta_N u-\Delta_{\Sigma}u)-\phi^2\nabla_{\nu_{\Sigma}} (hu)\\
			&+ \int_{\Sigma} \phi_t \langle\nabla u,\p_t \nu_{\Sigma_t}\rangle +\langle \nabla_{\Sigma} u, \nabla_{\Sigma} \phi \rangle \phi   -\int_{\partial \Sigma} u\phi \nabla_{\nu_{\partial \Sigma}} \phi
		\end{align*}
		Now we use that for a family of evolution of hypersurfaces using a normal vector field we have $-\nabla_{\Sigma} \phi=\partial_t \nu_{\Sigma}$ and from the Gauss equation we have $R_N=R_{\Sigma}+2\Ric(\nu,\nu)+|\sff|^2-H^2_{\Sigma}$ to get,
		\begin{align*}
			&\frac{d^2}{dt^2}\Big\vert_{t=0}(\mathcal{A}(\varphi_t(\Omega)))\\
			=& \int_{\Sigma} |\nabla_{\Sigma} \phi|^2u  -\frac{u\phi^2}{2}(R_{N}-R_{\Sigma}+|\sff|^2+H^2_{\Sigma})+ \phi^2 (\Delta_N u-\Delta_{\Sigma}u-\nabla_{\nu_{\Sigma}} (hu))\\
			&-\int_{\partial \Sigma} u\phi \langle \nabla_{\Sigma}\phi, \nu_{\partial \Sigma}\rangle\\
			=&\int_{\Sigma} |\nabla_{\Sigma} \phi|^2u  -\frac{u\phi^2}{2}(R_{N}-R_{\Sigma}+|\sff|^2+H^2_{\Sigma})+ \phi^2 (\Delta_N u-\Delta_{\Sigma}u-\nabla_{\nu_{\Sigma}} (hu))\\
			&+ \int_{\partial \Sigma} u\phi \langle \partial_t \nu_{\Sigma}, \nu_{\p \Sigma} \rangle ,\quad \text{at } t=0, \\
			=&\int_{\Sigma} |\nabla_{\Sigma} \phi|^2u  -\frac{u\phi^2}{2}(R_{N}-R_{\Sigma}+|\sff|^2+H^2_{\Sigma})+ \phi^2 (\Delta_N u-\Delta_{\Sigma}u-\nabla_{\nu_{\Sigma}} (hu))\\
			& + \int_{\partial \Sigma} u  \langle \nabla_{X_t} X_t, \nu_{\p \Sigma} \rangle ,\quad \text{at } t=0, \\
			=&\int_{\Sigma} |\nabla_{\Sigma} \phi|^2u  -\frac{u\phi^2}{2}(R_{N}-R_{\Sigma}+|\sff|^2+H^2_{\Sigma})+ \phi^2 (\Delta_N u-\Delta_{\Sigma}u-\nabla_{\nu_{\Sigma}} (hu))\\
			& - \int_{\partial \Sigma} u  \phi^2 A(\nu_{\Sigma},\nu_{\Sigma})
		\end{align*}
		
		We now write $|\sff|^2=|\mathring{\sff}|^2+\frac{H^2_{\Sigma}}{2}\geq  \frac{H^2_{\Sigma}}{2}$ 
		and notice that according to the first variation and $u>0$ we have $\frac{u^{-1}}{2}(uH_{\Sigma})^2=\frac{u^{-1}}{2}(\nabla_{\nu_{\Sigma}}u)^2+\frac{h^2u}{2}-h\nabla_{\nu_{\Sigma}}u$, so in total:

		\begin{align*}
			0\leq &\frac{d^2}{dt^2}\Big\vert_{t=0}(\mathcal{A}(\varphi_t(\Omega)))\\
			\leq& \int_{\Sigma} |\nabla \phi|^2u  -\frac{u\phi^2}{2}(R_{N}-R_{\Sigma})+\phi^2(-\frac{3H^2_{\Sigma}}{4}u+ \Delta_N u-\Delta_{\Sigma}u-\nabla_{\nu_{\Sigma}} (hu))\\
			& -\int_{\partial \Sigma} u  \phi^2 A(\nu_{\Sigma},\nu_{\Sigma})\\
			\leq & \int_{\Sigma} |\nabla \phi|^2u  -\frac{u\phi^2}{2}(R_{N}-R_{\Sigma})+\phi^2(\Delta_N u-\Delta_{\Sigma}u-u\nabla_{\nu_{\Sigma}} h-\frac{h^2u}{2}-\frac{u^{-1}}{2}(\nabla_{\nu_{\Sigma}}u)^2)\\
			& -\int_{\partial \Sigma} u  \phi^2 A(\nu_{\Sigma},\nu_{\Sigma})\\
		\end{align*}
		
	\end{proof}

	Combining the second variation of free boundary minmal hypersurface and that of $\mu-$bubble, we can produce a diameter bound as follows (see \cite{ChodoLiStryk2022-CompleteStableMinimal} for the case without boundary).

	\begin{theorem} \label{width}
		Consider $(X^4,\p X)$ a complete manifold with scalar curvature $R\geq 2, H_{\partial X \geq 0}$, and $(M,\p M) \hookrightarrow (X,\p X)$ a two-sided stable immersed free boundary minimal hypersurface.
		Let $N$ be a component of $\overline{M \setminus K}$ for some compact set $K$, with $\p N =\p_0 N \cup \p_1 N, \p_0 N \subset \p M$ and $\p_1 N \subset K$. If there is $p\in N$ with $d_N(p,\p_1 N)>2\pi,$ then we can find a Caccioppoli set $\Omega \subset B_{2\pi}(\p_1N)$ whose reduced boundary is smooth, so that any component $\Sigma$ of the reduced boundary $\p^{*}\Omega$ will have diameter at most $2\pi$ and intersect with $\p_0 N$ orthogonally.
	\end{theorem}
	
	\begin{remark}
		For convenience we also assume $\partial_1 N\cap \partial_0 N$ at angle $\theta \in (0,\pi/8)$ within $N$ due to similar regularity considerations as in secition \ref{pnp}. This  can be arranged by purturbing $N$ near an arbitrary small neighborhood, so will not influence the final bound for the diameter.
	\end{remark}

	\begin{proof}

		We again use $\sff$ for $N\hookrightarrow X $ and $A$ for $\partial X \hookrightarrow X$. We write  $\nu$ for the outward normal of $\partial N \subset N$ (the same for $\partial X \subset X$). 
		For any variation $\varphi_t$ of $(N,\partial N)$ compactly supported away from $\partial_1 N,$ writing $\frac{d}{dt} \big|_{t=0} \varphi_t= f\nu_{N} $, with $\nu_N$  a unit normal of $N \hookrightarrow X,$ we have by the second variation formula for stable free boundary minimal hypersurfaces:

		\begin{equation*}
			0\leq \frac{d^2}{dt^2}\big|_{t=0} \text{Area}(\varphi_t(N))=\int_N |\nabla_N f|^2 -(|\sff|^2+\Ric(\nu_N,\nu_N) )f^2-\int_{\partial_0 N} A(\nu_N,\nu_N)f^2.
		\end{equation*}

		Integration by parts gives us,
		\begin{equation*}
			0\leq \int_N-(f\Delta_N f+|\sff|^2f^2+\Ric(\nu_N,\nu_N)f^2) + \int_{\partial_0 N} f(\nabla_{\nu}f-A(\nu_N,\nu_N)f).
		\end{equation*}

		Using Lemma \ref{appen1} in Appendix, we can find a positive solution $u\in C^2(N)$ to $$\Delta_N u+(|\sff|^2+\Ric(\nu_N,\nu_N))u=0, \quad \nabla_{\nu}u-A(\nu_N,\nu_N)u=0 \text{ along } \partial_0 N.$$

		Now we follow Chodosh-Li-Stryer \cite{ChodoLiStryk2022-CompleteStableMinimal} and apply the free boundary $\mu$ bubble to the above $u$ and a proper $h$. 

		Consider a mollification of $d(\cdot,\partial_1 N)$ with Lipschitz constant less than 2, denoted as $\rho_0,$ 
		we may assume that $\rho_0(x)=0$ for all $x\in \partial_1 N$, and the level set $\{\rho_0(x)=2\pi\}$ is a smooth submanifold in $N$.
%		Choose $\epsilon \in (0,1/2)$ and that $\epsilon, 8\pi+2\epsilon$ are regular values of $\rho_0,$ we define $\rho$ with Lipschitz constant less than $1/4$,
%		\begin{equation*}
%			\rho=\frac{\rho_0-\epsilon}{8\pi+\epsilon/\pi}-\frac{\pi}{2.}
%		\end{equation*}

		Define $\Omega_1:=\{x\in N,0  <\rho_0<2\pi\}$, $\Omega_0:=\{0  <\rho_0<\pi\}$, and set $$h(x):=-\tan(\frac{1}{2}\rho_0(x)-\frac{\pi}{2})=-\tan(\rho(x)).$$ We solve the $\mu-$bubble problem among Caccioppoli sets whose symmetric difference with $\Omega_0$ is compact in $\Omega_1,$ i.e. we minimize the functional $\mathcal{A}(\Omega)$ using the given $h$ and $u>0$ obtained above.  We obtain a minimizer $\Omega$,  and for any component $\Sigma$ of $\partial^{*} \Omega$, we have $\partial \Sigma \cap \partial_0 N$ orthogonally and from the second variation formula in Theorem \ref{VarForm} we get for any compactly supported smooth function $\phi$ on $\Sigma$ (Lemma 15 of \cite{ChodoLi2020-GeneralizedSoapBubbles}),
		\begin{align*}
			0 \leq &\int_{\Sigma} |\nabla_{\Sigma} \phi|^2u - \frac{1}{2}(R_N-1-R_{\Sigma})\phi^2 u +(\Delta_N u-\Delta_{\Sigma}u) \phi^2 -\frac{1}{2u} (\nabla_{\nu_{\Sigma}}u)^2 \phi^2 \\
			&\int_{\Sigma} -\frac{1}{2}(1+h^2+2\nabla_{\nu_{\Sigma}}h)\phi^2 u-\int_{\partial \Sigma} A (\nu_{\Sigma},\nu_{\Sigma}) \phi^2 u,
		\end{align*}
		and now we have that $1+h^2+2\nabla_{\nu_{\Sigma}}h \geq 1+\tan^2(\rho)-\sec^2(\rho) =1-1=0.$ So in total we have:
		\begin{align*}
			0 \leq &\int_{\Sigma} |\nabla_{\Sigma} \phi|^2u - \frac{1}{2}(R_N-1-R_{\Sigma})\phi^2 u +(\Delta_N u-\Delta_{\Sigma}u) \phi^2 -\frac{1}{2u} (\nabla_{\nu_{\Sigma}}u)^2 \phi^2 \\
			&-\int_{\partial \Sigma} A(\nu_{\Sigma},\nu_{\Sigma}) \phi^2 u.
		\end{align*}

		We can plug in the equation (\ref{eigen1}) for $u$, using $R_g \geq 2$ and Gauss Equation $R_X=R_N +2\Ric_X(\nu_N,\nu_N)+|\sff_N|^2-H^2_N$ to get:
		\begin{align*}
			0 \leq &\int_\Sigma |\nabla_\Sigma \phi|^2 u -\frac{1}{2}(1-R_{\Sigma})\phi^2u-\Delta_{\Sigma}u \phi^2 -\frac{1}{2u} (\nabla_{\nu_{\Sigma}}u)^2 \phi^2 -\int_{\partial \Sigma} A (\nu_{\Sigma},\nu_{\Sigma}) \phi^2 u \\
			\leq & \int_{\Sigma} |\nabla_{\Sigma} \phi|^2 u -(\frac{1}{2}-K_{\Sigma}) \phi^2 u -\Delta_{\Sigma} u \phi^2 -\int_{\partial \Sigma} A (\nu_{\Sigma},\nu_{\Sigma}) \phi^2 u\\
			\leq & \int_{\Sigma} -\dive(u\nabla_{\Sigma} \phi) \phi -(\frac{1}{2}-K_{\Sigma})  \phi^2 u  -\Delta_{\Sigma} u \phi^2 - \int_{\partial \Sigma} (A (\nu_{\Sigma},\nu_{\Sigma}) \phi-\langle \nabla_{\Sigma}\phi, \eta\rangle) \phi u
		\end{align*}

		By the same argument we used above to obtain $u,$ we can find a function  $w$ with $A (\nu_{\Sigma},\nu_{\Sigma})w-\langle \nabla_{\Sigma}w, \eta\rangle=0$ so that on $\Sigma$,
		\begin{equation*}
			\dive(u\nabla_{\Sigma} w)  +(\frac{1}{2}-K_{\Sigma})uw+ w\Delta_{\Sigma} u = 0
		\end{equation*}

		We let $f=uw$ and by combining the equation above we have over $\Sigma$: 
		\begin{align*}
			\Delta_{\Sigma} f & = w\Delta_{\Sigma} u +\dive_{\Sigma}(u\nabla_{\Sigma} w)+\nabla_{\Sigma} u \cdot \nabla_{\Sigma} w \\
			&= -(\frac{1}{2}-K_{\Sigma})uw+ \nabla u \cdot \nabla w \\
			&= -(\frac{1}{2}-K_{\Sigma})uw+ \frac{1}{2uw}(|\nabla f|^2-u|\nabla w|^2-w|\nabla u|^2)\\
			&\leq -(\frac{1}{2}-K_{\Sigma}) f+\frac{1}{2f} |\nabla f|^2.
		\end{align*}
		 
	Lemma 17 in Chodosh-Li \cite{ChodoLi2020-GeneralizedSoapBubbles} also holds under the following condition (a short proof is obtained in the Appendix, see Lemma \ref{appen2}):
		 \begin{equation*}
		 	\partial_{\eta}f=u\partial_{\eta}w+w\partial_{\eta}u=A(\nu_{\Sigma},\nu_{\Sigma})uw+A(\nu_N,\nu_N)uw=H_{\partial X}f-k_{\partial\Sigma}f\geq -k_{\partial \Sigma}f.
		 \end{equation*}
		 So diam$(\Sigma) \leq 2\pi$.
	\end{proof}

	\begin{theorem}[Almost Linear Growth of An End] \label{linearvol}
		Let $(X^4, \partial X)$ be a complete manifold with weakly bounded geometry, $H_{\partial X}\geq 0, \Ric_2 \geq 0$ and $R_g \geq 2.$ Let $(M^3,\partial M) \hookrightarrow (X^4,\partial X)$ be a complete simply connected two-sided stable free boundary minimal immersion. Let $(E_k)_{k \in \N}$ be an end of $M$ given by $E_k =M \setminus B_{kL}(x)$ for some fixed point $x \in M$ and let $M_k:= E_k \cap \overline{B_{(k+1)L}(x)},$ here $L=20\pi$ (determined by the constant in the lemma above). Then there is a constant $C_0=C(X,L)$ and $k_0$ such that for $k \geq k_0$, 
		\begin{equation*}
			\text{Vol}_M(M_k) \leq C_0.
		\end{equation*} 
	\end{theorem}

	\begin{proof}
		The proof that there is a large $k_0$ so that for all $k\geq k_0$, $M_k$ is connected is the same as \cite{ChodoLiStryk2022-CompleteStableMinimal} Proposition 3.2 (this uses the simply-connectedness). For each $E_k$ we can purturb the boundary so that it intersects with $\partial M$ with an interior angle $\theta \in (0,\pi/8)$ and we can apply Theorem (\ref{width}) to $E_k \hookrightarrow X$, so we obtain $\Omega_k \subset B_{\frac{L}{2}}(\partial E_k)$. Also with the same proof as \cite{ChodoLiStryk2022-CompleteStableMinimal} Lemma 5.4, there is some component $\Sigma_k$ of $\partial \Omega_k$ that separates $\partial E_k$ and $\partial E_{k+1}$, then Theorem (\ref{width}) implies that diam($\Sigma_k$) $\leq c $ for $(c=2\pi)$ and diam$(M_k)  \leq 4L+c.$ We can show this last inequality by taking any two points $z_1, z_2$ in $M_k$, for each $z_i$ there is a minimizing path connecting $x$ and $z_i$ and intersecting $\Sigma_k$ at some point $y_i$, the arc connecting $y_i,z_i$ is at most $2L$ and combining with diam$(\Sigma_k) \leq c$ we get $d(z_1,z_2) \leq 4L+c$.

		Now by curvature estimates Lemma \ref{blowup} we can apply the volume control Lemma \ref{WBGvol}, to get a constant $C_0=C(X,g, L, c)$ such that,
		\begin{equation*}
			\text{Vol}(B_{4L+c}(p)) \leq C_0,
		\end{equation*}
		for all $p \in M$. Since diam$(M_k)\leq 4L+c$, we get Vol$(M_k) \leq C_0$ as desired.
	\end{proof}

	\section{Proof of Main Theorem and Necessity of Convexity Assumption}

	Now we are ready to prove the main theorem. We first explain some set up. 
	
	We first assume $M$ is simply connected and has infinte volume (otherwise the proof is the same as assuming $M$ is compact as described in the introduction), and by section \ref{1NonparaEnd} we know $M$ has at most 1 nonparabolic end $(E_k)_{k \in \N}$ which we can apply Theorem \ref{linearvol} to obtain $M_k$ and $k_0, L,c$ following the notation in Theorem (\ref{linearvol}). 

	We write $M$ as a decomposition of the following components, fixing $x \in M$ and write $B_R(x)$ as $B_R$, 
	\begin{align*}
		M &=\overline{B_{k_0 L}} \cup E_{k_0} \cup (M\setminus (B_{k_0L} \cup E_{k_0}))\\
		&=:\overline{B_{k_0 L}} \cup E_{k_0} \cup P_{k_0}\\
	\end{align*}

	We also have inductively,for each $i\geq 1$:
	\begin{align*}
		E_{k_0}&=M_{k_0} \cup P_{k_0+1} \cup E_{k_0+1}\\
		&=M_{k_0} \cup P_{k_0+1} \cup (M_{k_0+1} \cup P_{k_0+2}  \cup E_{k_0+2})\\
		&=\left(\bigcup \limits_{k=k_0}^{k_0+i-1} M_k \right)\cup \left(\bigcup \limits_{k=k_0+1}^{k_0+i} P_k\right) \cup E_{k_0+i}
	\end{align*}
	where each $P_k$ when $k > k_0$ is defined as $E_{k} \setminus (E_{k+1} \cup B_{(k+1)L})$, and each component of $P_k$ for $k \geq k_0$ is parabolic.

	We restate the main theorem for convenience of reader:

	\begin{theorem}\label{ThmMain}
		Let $(X^4, \partial X)$ be complete  with $R_g \geq 2$, $Ric_2 \geq 0$, weakly bounded geometry and weakly convex boundary. Then any complete stable two-sided free boundary minimal hypersurface $(M^3,\partial M)$ is totally geodesic and $Ric(\eta,\eta)=0$ along $M$ and $A(\eta,\eta)=0$ along $\p M$, for $\eta$ a choice of normal bundle over $M$.
	\end{theorem}

	\begin{proof}
		Following the set up above, fix $x \in M$, $i \geq 1$ and obtain $k_0, L,c, E_k, M_k, P_k$.

		For each $k \geq k_0$, $P_k$ is made of disjoint parabolic components. $P_{k_1}$ and $P_{k_2}$ are also disjoint if $k_1 \neq k_2$. So we can apply Lemma (\ref{parabolic0energy}) to each of these component, and obtain a compactly supported function $u_k$ on each $P_k$, with $\int_{P_k} |\nabla u_{k}|^2 < \frac{1}{i^2}$ and with the boundary condition $u_k \vert_{\partial P_k \setminus \partial M}=1, \nabla_{\nu}u\vert_{\partial M \cap P_k}=0$.

		We let $\rho$ a mollification of the distance function to $x$, with $|\nabla \rho|\leq 2$ and
		\begin{equation*}
			\rho\vert_{\partial E_k} = kL, \rho \vert_{\partial M_k \setminus \partial E_k}=(k+1)L.
		\end{equation*}
		Consider $\phi(x)=\frac{(k_0+i)L-x}{iL},$ then we can define a compactly supported Lipschitz function $f_i$ as follows. When $x\in \overline{M_k}$ for some $k_0 \leq k \leq k_0+i-1$, then $f_i(x)= \phi(\rho(x))$, and when $x\in \overline{P_k}$ for some $k_0 \leq k \leq k_0+i$ we define $f(x)=\phi(kL)u_k.$ One can check that this definition agrees on the intersection, and we can define $f(x)=1$ when $x\in \overline{B_{k_0L}},$ and $f(x)=0$ when $x \in E_{k_0+i}.$ Now we can apply this test function into the stability inequality for free boundary minimal hypersurface, together with $A\geq 0$:
		\begin{align*}
			\int_M (\Ric(\eta,\eta)+|\sff|^2)f_i^2&\leq \int_M |\nabla f_i|^2-\int_{\partial M} A(\eta,\eta)f^2\\
			&\leq\sum_{k=k_0}^{k_0+i-1} \int_{M_k} \phi'(\rho)^2|\nabla \rho|^2+\sum_{k=k_0}^{k_0+i} \phi^2(kL) \int_{P_k} |\nabla u_k|^2  \\
			&\leq \frac{4iC_0}{i^2L^2}+ \frac{i+1}{i^2} \leq \frac{C'}{i} \rightarrow 0 \quad \text{as } i \rightarrow \infty.
		\end{align*}
		Since $f_i \rightarrow 1$ on $M$ as we let $i\rightarrow \infty,$ 
		we get that everywhere on $M$, $\Ric(\eta,\eta)=0$ and $\sff=0$, and $A(\eta,\eta)=0$ along $\p M$. 
	\end{proof}

	We note that until the last step, $A_2 \geq 0$ is sufficient. We now provide a counterexample to Theorem \ref{ThmMain} if one replace $A \geq 0$ by $A_2 \geq 0$.

	Consider $\Sph^4 \subset \R^5$ with induced metric, and any closed curve $\gamma \subset \Sph^4$, we look at the intrinsic neighborhood $X=B_{\epsilon}(\gamma):=\{x \in \Sph^4, d(x,\gamma)\leq \epsilon\}$. We can choose $\gamma$ so that $A_2 \geq 0$ everywhere but $A(e_1,e_1)<0$ for some nonzero $e_1$ at a point in $X$. We can minimize area among all hypersurfaces with (nonempty) boundary and nontrivial homology class contained in $\p X$, then we have a stable free boundary minimal immersion.

\section{Appendix}
\begin{lemma}\label{appen1}
		Assume as in Theorem \ref{width}, we have over $(N,\p N=\p_0 N\cup \p_1 N)$ the following second variation formula for any $\phi\in C^{\infty}_c(\overline{\p N \setminus \p_1 N})$,
		\begin{equation*}
			0\leq \int_N -(\Delta_N \phi+(|\sff|^2+\Ric(\nu_N,\nu_N))\phi)\phi+\int_{\p_0 N}(\nabla_{\nu}\phi-A(\nu_N,\nu_N)\phi)\phi,
		\end{equation*}
		then there is a $C^{2}$ solution to $\Delta_N u+(|\sff|^2+\Ric(\nu_N,\nu_N))u=0$, $\nabla_{\nu} u-A(\nu_N,\nu_N)u=0$ along $\partial_0 N$,  and $u\rvert_{N^{\circ}}>0$. 
\end{lemma}

\begin{proof}
		We denote the first eigenvalue as:
		\begin{equation*}
			\lambda_1(N):=\inf_S \frac{\int_N |\nabla \phi|^2-(|\sff_N|^2 +\Ric(\nu_N,\nu_N))\phi^2-\int_{\p_0 N}A(\nu_N,\nu_N)\phi^2}{\int_N \phi^2},
		\end{equation*}
		where $S=\{\phi \in C^{\infty}_c({ N \setminus \p_1 N}), \phi\neq 0\}$.
		
		Note since each test function $\phi$ is taken to be compactly supported, $\lambda_1(N)\geq 0$ is well-defined by domain monotonicity property for compact sets (that is, $\lambda_1(B_1) > \lambda_1(B_2) $ when $\overline{B_1}\subset B_2 \Subset N $) from Fischer-Colbrie and Schoen \cite{FischSchoe1980-StructureCompleteStable}.

		Following Theorem 6.15 in \cite{GilbaTrudi2001-EllipticPartialDifferential}, we consider the following problem over a compact exhaustion $(\Omega_l)_{l\in \N}$ of $N$, each containing the boundary $\partial_1 N$:
		\begin{align} \label{appen3}
			(\Delta_N +|\sff_N|^2+\Ric(\nu_N,\nu_N))\varphi=0, \quad & \Omega_l^{\circ}  \\
			\nabla_{\nu}\varphi-A(\nu_N,\nu_N)\varphi=0, \quad & \partial_0 N \cap  \Omega_l\nonumber\\ 
			\varphi=1, \quad & \partial^{*} \Omega_l\nonumber\\
			\varphi=0, \quad & \partial_1 N.\nonumber
		\end{align}
		We can make the boundary condition homogeneous by subtracting a function $g_l\in C^3(N)$ satisfying the boundary conditions in (\ref{appen3}), then we claim the following problem has a unique solution $\varphi'=\varphi-g_l\in C^3(N)$, 
	\begin{align} \label{appen4}
			(\Delta_N +|\sff_N|^2+\Ric(\nu_N,\nu_N))\varphi'=g_l', \quad & \Omega_l^{\circ}  \\
			\nabla_{\nu}\varphi-A(\nu_N,\nu_N)\varphi'=0, \quad & \partial_0 N \cap  \Omega_l\nonumber\\ 
			\varphi'=0, \quad & \partial^{*} \Omega_l\nonumber\\
			\varphi'=0, \quad & \partial_1 N,\nonumber
		\end{align}
		with $g'_l:=(\Delta_N+|\sff_N|^2+\Ric(\nu_N,\nu_N))g_l.$

	Indeed, consider the operator $L_l:=\Delta_N+|\sff_N|^2+\Ric(\nu_N,\nu_N)-c_0$ over $\Omega_l$ where $c_0 \gg \max\{\sup_{\Omega_l}(|\sff_N|^2+\Ric(\nu_N,\nu_N)),\sup_{\Omega_l} A(\nu_N,\nu_N),0\}$ and the corresponding bilinear form, 
	\begin{equation*}
		B_l(h,k):=\int_{\Omega_l} \nabla h\cdot \nabla k+(c_0-|\sff_N|^2-\Ric(\nu_N,\nu_N))hk-\int_{\p \Omega_l}A(\nu_N,\nu_N)hk.
	\end{equation*}

	By domain monotonicity, $\lambda_1(N)\geq 0$ implies $\lambda_1(\Omega_l)>0$, so that  $B_l(h,h)\geq \alpha_1\int_{\Omega_l}h^2$ for any $h\in S\cap C^{\infty}_c(\Omega_l)$, for some $\alpha_1 >0$. By the choice of $c_0$ and the multiplicative trace inequality (see \cite{brenner2008mathematical} Theorem 1.6.6), we also have $B_l(h,h)\geq \alpha_2 \int_{\Omega_l}|\nabla h|^2 $ for any $h\in S\cap C^{\infty}_c(\Omega_l)$, for some $\alpha_2 >0$. Together this implies that $B(\cdot, \cdot)$ is coercive in the $W^{1,2}(\Omega_l)$ norm for functions in $S$. Then by Lax-Milgram and Schauder regularity theory, there is a unique solution $f_l\in C^{2,\alpha}(\Omega_l)$ to $L_lf_l=g''_l$ for any $g''_l \in S\cap C^{\infty}_c(\Omega_l)$, with the boundary condition, $\nabla_{\nu_{\p_0 N}}f_l=A(\nu_N,\nu_N)f_l$ along $\p_0 N \cap \Omega_l$. 
	
	We denote $S'=\{f\in C^{2,\alpha}(\Omega_l), f\rvert_{\p_1 N}=0, \nabla_{\nu_{\p_0N}}f=A(\nu_N,\nu_N)f\text{ along } \p_0N\}$ then $L_l$ is invertible from $S'$ to $C^{\alpha}(\Omega_l)$ and the inverse $L^{-1}_l$ is a compact operator from $C^{\alpha}(\Omega_l)$ to $C^{\alpha}(\Omega_l)$. By Theorem 5.3 in \cite{GilbaTrudi2001-EllipticPartialDifferential}, we have that the operator $T:=\text{Id}+c_0L_l^{-1}$ satisfies the Fredholm alternative on the normed vector space $C^{\alpha}(\Omega_l)$.
	
	Now we have,
	\begin{equation*}
		Tu:=(\text{Id}+c_0L^{-1}_l)u=L^{-1}_lf \iff L_N u:=\Delta_N u+(|\sff|^2+\Ric(\nu_N,\nu_N))u=f,
	\end{equation*}
	and note $u=Tu-c_0L^{-1}_lu=L^{-1}_lf-c_0L^{-1}_lu\in S'$.
	
	So using $\lambda_1(\Omega_l)>0$, we have that there is no non-trivial solution to $L_Nu=0$ in $S'$, hence by Fredholm alternative, there is a unique solution to (\ref{appen4}) given $g_l'$ and the claim is proved. We denote the unique solution to (\ref{appen3}) as $v_l$.

	We claim that each $v_l>0$ on $\Omega_l^{\circ}$. By Hopf Lemma (\cite{GilbaTrudi2001-EllipticPartialDifferential} Theorem 3.5), it's enough to show $v_l \geq 0$.

	Note since $v_l\in H^1(\Omega_l)$ and $|\cdot|$ is $1$-Lipschitz, $|v|\in H^1(\Omega_l)$ with $\nabla |v_l|=\text{sign}(v_l)\nabla v_l$ almost everywhere, so both $v_l$ and $|v_l|$ minimizes the first eigenvalue problem and $|v_l|\in C^{2,\alpha}(\Omega_l)$. Now assume $\{x\in \Omega^{\circ}_l, v_l(x) <0\} \neq \emptyset$, then by maximum principle (\cite{GilbaTrudi2001-EllipticPartialDifferential} Theorem 3.5) $|v_l|$ obtain minimum at some point $x_0\in \Omega_l$ with $v_l(x_0)=0$, then $|v_l|\equiv 0$ on $\Omega_l$, a contradiction to the boundary conditions in (\ref{appen3}).

	Now we have that $v_l\rvert_{\Omega^{\circ}_l}>0$, fix a point $p\in \Omega_1^{\circ} \subset N\setminus \p N$,  and denote $u_l(x) :=\frac{v_l(x)}{v_l(p)}$ then we can proceed as in \cite{FischSchoe1980-StructureCompleteStable},  Harnack inequality gives $u_l$ subsequentially converge in $C^{2}_{\loc}(N)$ to a nonzero function on $N,$ with $u>0$ on $N^{\circ}$, $u(p)=1, u\rvert_{\p_1 N}=0$ and,
		\begin{equation}\label{eigen1}
			(\Delta_N +Ric_X(\nu_N,\nu_N)+|\sff_N|^2)u = 0 \quad \text{on } N^{\circ}, \quad \nabla^N_{\nu}u-A(\nu_N,\nu_N)u=0 \quad \text{on } \partial_0 N.
		\end{equation}
\end{proof}

\begin{lemma}\label{appen2}
		If $(\Sigma^2,\partial \Sigma,g)$ is a compact Riemannian manifold with Gauss curvature $K_{\Sigma}$ and,
		\begin{equation}\label{appen}
			\Delta_{\Sigma} \lambda \leq -(K_0-K_{\Sigma})\lambda+\frac{|\nabla_{\Sigma}\lambda|^2}{2\lambda}, \quad\nabla_{\eta}\lambda +k_{\partial \Sigma}\lambda \geq 0
		\end{equation}
		for some smooth $\lambda>0$, $\eta$ the outward unit normal of  $\partial\Sigma \subset \Sigma$, $k_{\partial \Sigma}$ the corresponding geodesic curvature and $K_{0}\in(0,\infty)$. Then $\text{diam}_g\Sigma\leq \sqrt{\frac{2}{K_0}}\pi.$
\end{lemma}

\begin{proof}
	We follow the proof of Lemma 16 and Lemma 17 in \cite{ChodoLi2020-GeneralizedSoapBubbles} and track the sign of the boundary terms carefully. If not, then we can find a free boundary curve $\gamma:[a,b]\rightarrow \Sigma$ with $\partial \gamma\subset \partial \Sigma$ and the following (from Proposition 15 in \cite{ChodoLi2020-GeneralizedSoapBubbles}), take $u=\lambda$ and $\psi^2u=1$,
	\begin{align*}
		0&\leq \int_{\gamma}|\nabla_{\gamma} \psi|^2u-\frac{1}{2}(R_{\Sigma}-2K_0)\psi^2 u+(\Delta_{\Sigma} u-\Delta_{\gamma}u)\psi^2-\frac{1}{2}(2K_0+h^2+2\nabla_{\nu_{\gamma}}h)\psi^2u\\
		&-\int_{\gamma}\frac{|\nabla_{\nu_{\gamma}}u|^2}{2u^2} -\int_{\partial\gamma}\sff_{\partial \Sigma}(\nu_{\gamma},\nu_{\gamma})\psi^2 u \\
		&=\int_{\gamma}\frac{1}{4u^2}|\nabla_{\gamma}u|^2-\frac{1}{2}(R_{\Sigma}-2K_0)+u^{-1}(\Delta_{\Sigma}u-\Delta_{\gamma}u)-\frac{1}{2}(2K_0+h^2+2\nabla_{\nu_{\gamma}}h)\\
		&-\int_{\gamma}\frac{|\nabla_{\nu_{\gamma}}u|^2}{2u^2} -\int_{\partial\gamma}\sff_{\partial \Sigma}(\nu_{\gamma},\nu_{\gamma})\\
		&\stackrel{(\star 1)}{<}\int_{\gamma} \frac{1}{4u^2}|\nabla_{\gamma}u|^2+\frac{|\nabla_{\Sigma}u|^2-|\nabla_{\nu_{\gamma}}u|^2}{2u^2}-u^{-1}\Delta_{\gamma}u-\int_{\partial\gamma}\sff_{\partial \Sigma}(\nu_{\gamma},\nu_{\gamma})\\
		&\stackrel{(\star 2)}{=}\int_{\gamma}\frac{3}{4u^2}|\nabla_{\gamma}u|^2+\nabla_{\gamma}(u^{-1})\nabla_{\gamma}u-\int_{\partial \gamma}u^{-1}\nabla_{\nu_{\partial\gamma}}u+\sff_{\partial \Sigma}(\nu_{\gamma},\nu_{\gamma})\\
		&\stackrel{(\star 3)}{=}\int_{\gamma}\frac{-1}{4u^2}|\nabla_{\gamma}u|^2- \int_{\partial \gamma}u^{-1}\nabla_{\eta}u+k_{\partial \Sigma}\leq  0,
	\end{align*}
	where in $(\star 1)$ we used the assumption (\ref{appen}) and that $(K_0+\frac{1}{2}h^2+\nabla_{\nu_{\gamma}}h)>0$ as in Lemma 16 of \cite{ChodoLi2020-GeneralizedSoapBubbles}; in $(\star 2)$ we used integration by parts; in $(\star 3)$ we used $\nu_{\partial \gamma}=\eta$ by free boundary. The strict inequality gives a contradiction.  
\end{proof}

\bibliographystyle{alpha}
\bibliography{FBMHPSC.bib}

\begin{thebibliography}{ACM19}

\bibitem[ACM19]{AmbroCarloMassa2018-LecturesEllipticPartial}
Luigi Ambrosio, Alessandro Carlotto, and Annalisa Massaccesi.
\newblock {\em Lectures on elliptic partial differential equations}, volume~18.
\newblock Springer, 2019.

\bibitem[AK82]{AzzamKreys1982-SolutionsEllipticEquations}
A.~Azzam and E.~Kreyszig.
\newblock On solutions of elliptic equations satisfying mixed boundary conditions.
\newblock {\em SIAM J. Math. Anal.}, 13(2):254--262, 1982.

\bibitem[Azz81]{azzam1981smoothness}
A.~Azzam.
\newblock Smoothness properties of solutions of mixed boundary value problems for elliptic equations in sectionally smooth {$n$}-dimensional domains.
\newblock {\em Ann. Polon. Math.}, 40(1):81--93, 1981.

\bibitem[BS08]{brenner2008mathematical}
Susanne~C Brenner and L~Ridgway Scott.
\newblock {\em The mathematical theory of finite element methods}.
\newblock Springer, 2008.

\bibitem[BZ19]{BamleZhang2015-HeatKernelCurvature}
Richard~H. Bamler and Qi~S. Zhang.
\newblock Heat kernel and curvature bounds in {R}icci flows with bounded scalar curvature---{P}art {II}.
\newblock {\em Calc. Var. Partial Differential Equations}, 58(2):Paper No. 49, 14, 2019.

\bibitem[CL23]{ChodoLi2022-StableAnisotropicMinimal}
Otis Chodosh and Chao Li.
\newblock Stable anisotropic minimal hypersurfaces in $\mathbf{R}^4$.
\newblock {\em Forum Math. Pi}, 11:Paper No. e3, 22, 2023.

\bibitem[CL24a]{ChodoLi2020-GeneralizedSoapBubbles}
Otis Chodosh and Chao Li.
\newblock Generalized soap bubbles and the topology of manifolds with positive scalar curvature.
\newblock {\em Ann. of Math. (2)}, 199(2):707--740, 2024.

\bibitem[CL24b]{ChodoLi2021-StableMinimalHypersurfaces}
Otis Chodosh and Chao Li.
\newblock Stable minimal hypersurfaces in {${\bf R}^4$}.
\newblock {\em Acta Math.}, 233(1):1--31, 2024.

\bibitem[CLS24]{ChodoLiStryk2022-CompleteStableMinimal}
Otis Chodosh, Chao Li, and Douglas Stryker.
\newblock Complete stable minimal hypersurfaces in positively curved 4-manifolds.
\newblock {\em Journal of the European Mathematical Society}, 2024.

\bibitem[CMR24]{CatinMastrRonco2023-TwoRigidityResults}
Giovanni Catino, Paolo Mastrolia, and Alberto Roncoroni.
\newblock Two rigidity results for stable minimal hypersurfaces.
\newblock {\em Geom. Funct. Anal.}, 34(1):1--18, 2024.

\bibitem[DPM15]{DePhMaggi2014-RegularityFreeBoundaries}
G.~De~Philippis and F.~Maggi.
\newblock Regularity of free boundaries in anisotropic capillarity problems and the validity of {Y}oung's law.
\newblock {\em Arch. Ration. Mech. Anal.}, 216(2):473--568, 2015.

\bibitem[DPM17]{DePhMaggi2014-DimensionalEstimatesSingular}
Guido De~Philippis and Francesco Maggi.
\newblock Dimensional estimates for singular sets in geometric variational problems with free boundaries.
\newblock {\em J. Reine Angew. Math.}, 725:217--234, 2017.

\bibitem[Eva22]{Evans2010-PartialDifferentialEquations}
Lawrence~C Evans.
\newblock {\em Partial differential equations}, volume~19.
\newblock American Mathematical Society, 2022.

\bibitem[FCS80]{FischSchoe1980-StructureCompleteStable}
Doris Fischer-Colbrie and Richard Schoen.
\newblock The structure of complete stable minimal surfaces in 3-manifolds of non-negative scalar curvature.
\newblock {\em Communications on Pure and Applied Mathematics}, 33(2):199--211, 1980.

\bibitem[GJ86]{GruteJost1986-AllardTypeRegularity}
Michael Gr\"uter and J\"urgen Jost.
\newblock Allard type regularity results for varifolds with free boundaries.
\newblock {\em Ann. Scuola Norm. Sup. Pisa Cl. Sci. (4)}, 13(1):129--169, 1986.

\bibitem[GLZ20]{GuangLiZhou2020-CurvatureEstimatesStable}
Qiang Guang, Martin Man-chun Li, and Xin Zhou.
\newblock Curvature estimates for stable free boundary minimal hypersurfaces.
\newblock {\em J. Reine Angew. Math.}, 759:245--264, 2020.

\bibitem[GP10]{GuillPolla1974-DifferentialTopology}
Victor Guillemin and Alan Pollack.
\newblock {\em Differential topology}, volume 370.
\newblock American Mathematical Soc., 2010.

\bibitem[GT77]{GilbaTrudi2001-EllipticPartialDifferential}
David Gilbarg and Neil~S Trudinger.
\newblock {\em Elliptic partial differential equations of second order}.
\newblock Springer, 1977.

\bibitem[Lee03]{Lee2012-IntroductionSmoothManifoldsa}
John~M Lee.
\newblock {\em Smooth manifolds}.
\newblock Springer, 2003.

\bibitem[Lie86]{Liebe1986-MixedBoundaryValue}
Gary~M Lieberman.
\newblock Mixed boundary value problems for elliptic and parabolic differential equations of second order.
\newblock {\em Journal of Mathematical Analysis and Applications}, 113(2):422--440, 1986.

\bibitem[Mag12]{Maggi2012-SetsFinitePerimeter}
Francesco Maggi.
\newblock {\em Sets of finite perimeter and geometric variational problems: an introduction to Geometric Measure Theory}, volume 135.
\newblock Cambridge University Press, 2012.

\bibitem[Mir55]{Miran1955-SulProblemaMisto}
Carlo Miranda.
\newblock Sul problema misto per le equazioni lineari ellittiche.
\newblock {\em Ann. Mat. Pura Appl. (4)}, 39:279--303, 1955.

\bibitem[Mir13]{Miran1970-PartialDifferentialEquations}
Carlo Miranda.
\newblock {\em Partial differential equations of elliptic type}.
\newblock Springer-Verlag, 2013.

\bibitem[SY79a]{SchoeYau1979-StructureManifoldsPositive}
R~Schoen and ST~Yau.
\newblock On the structure of manifolds with positive scalar curvature.
\newblock {\em manuscripta mathematica}, 28(1):159--183, 1979.

\bibitem[SY79b]{SchoeYau1979-ExistenceIncompressibleMinimal}
Richard Schoen and Shing-Tung Yau.
\newblock Existence of incompressible minimal surfaces and the topology of three dimensional manifolds with non-negative scalar curvature.
\newblock {\em Annals of Mathematics}, 110(1):127--142, 1979.

\bibitem[SY96]{ShenYe1996-StableMinimalSurfaces}
Ying Shen and Rugang Ye.
\newblock On stable minimal surfaces in manifolds of positive bi-ricci curvatures.
\newblock {\em Duke Mathematical Journal}, 1996.

\end{thebibliography}

\end{document}